\documentclass[12pt,reqno]{amsart}

\usepackage{amsmath,amsfonts,amsthm,amssymb,amsxtra,bbm}
\usepackage{fourier,setspace,graphicx,xcolor}
\usepackage[colorlinks=true]{hyperref}
\hypersetup{urlcolor=blue, linkcolor=blue, citecolor=red, anchorcolor=blue]}

\newtheorem{theorem}{Theorem}
\newtheorem{proposition}[theorem]{Proposition}
\newtheorem{lemma}[theorem]{Lemma}
\newtheorem{corollary}[theorem]{Corollary}
\theoremstyle{definition}

\theoremstyle{remark}
\newtheorem{remark}[theorem]{Remark}

\newcommand\version{\today}
\newcommand{\be}[1]{\begin{equation}\label{#1}}
\newcommand{\ee}{\end{equation}}
\newcommand{\1}{\mathbbm{1}}

\renewcommand{\epsilon}{\varepsilon}

\newcommand{\F}{\mathcal{F}}
\renewcommand{\P}{\mathcal{P}}

\newcommand{\N}{\mathbb{N}}

\renewcommand{\phi}{\varphi}
\newcommand{\R}{\mathbb{R}}

\newcommand{\Sph}{\mathbb{S}}

\newcommand{\irN}[1]{\int_{\R^N}{#1}\,dx}
\newcommand{\Jlambda}[2]{\int_{\R^N}|#2|^\lambda\,#1(#2)\,d{#2}}

\newcommand{\measurerestr}{\,\raisebox{-.127ex}{\reflectbox{\rotatebox[origin=br]{-90}{$\lnot$}}}\,}

\reversemarginpar
\usepackage[colorinlistoftodos,prependcaption,textsize=tiny]{todonotes}

\begin{document}
\title[Reverse HLS --- \version]{Reverse Hardy--Littlewood--Sobolev inequalities}

\author[J.~A.~Carrillo]{Jos\'e A. Carrillo}
\address{\hspace*{-12pt}J.~A.~Carrillo: Department of Mathematics, Imperial College London, London SW7 2AZ, UK}
\email{\href{mailto:carrillo@imperial.ac.uk}{carrillo@imperial.ac.uk}}

\author[M.~G.~Delgadino]{Mat\'ias G.~Delgadino}
\address{\hspace*{-12pt}M.~G.~Delgadino: Department of Mathematics, Imperial College London, London SW7 2AZ, UK}
\email{\href{mailto:m.delgadino@imperial.ac.uk}{m.delgadino@imperial.ac.uk}}

\author[J.~Dolbeault]{Jean Dolbeault}
\address{\hspace*{-12pt}J.~Dolbeault (corresponding author): CEntre de REcherche en MAth\'ematiques de la D\'Ecision (CNRS UMR n$^\circ$ 7534), PSL research university, Universit\'e Paris-Dau\-phine, Place de Lattre de Tassigny, 75775 Paris 16, France}
\email{\href{mailto:dolbeaul@ceremade.dauphine.fr}{dolbeaul@ceremade.dauphine.fr}}

\author[R.~Frank]{Rupert L.~Frank}
\address{\hspace*{-12pt}R.~L.~Frank: Mathematisches Institut, Ludwig-Maximilians Universit\"at M\"unchen, Theresienstr. 39, 80333 M\"unchen, Germany, and Department of Mathematics, California Institute of Technology, Pasadena, CA 91125, USA}
\email{\href{mailto:r.frank@lmu.de}{r.frank@lmu.de}}

\author[F.~Hoffmann]{Franca Hoffmann}
\address{\hspace*{-12pt}F.~Hoffmann: Department of Computing and Mathematical Sciences, California Institute of Technology, 1200 E California Blvd. MC 305-16, Pasadena, CA 91125, USA}
\email{\href{mailto:fkoh@caltech.edu}{fkoh@caltech.edu}}

\begin{abstract} This paper is devoted to a new family of reverse Hardy--Littlewood--Sobolev inequalities which involve a power law kernel with positive exponent. We investigate the range of the admissible parameters and the properties of the optimal functions. A striking open question is the possibility of concentration which is analyzed and related with free energy functionals and nonlinear diffusion equations involving mean field drifts.
\\[4pt]
{\sc R\'esum\'e.} Cet article est consacr\'e \`a une nouvelle famille d'in\'egalit\'es de Hardy--Little\-wood--Sobolev invers\'ees correspondant \`a un noyau en loi de puissances avec un exposant positif. Nous \'etudions le domaine des param\`etres admissibles et les propri\'et\'es des fonctions optimales. Une question ouverte remarquable est la possibilit\'e d'un ph\'e\-no\-m\`ene de concentration, qui est analys\'e est reli\'e \`a des fonctionnelles d'\'energie libre et \`a des \'equations de diffusion non-lin\'eaires avec termes de d\'erive donn\'es par un champ moyen.
\end{abstract}

\keywords{Reverse Hardy--Littlewood--Sobolev inequalities; interpolation; symmetrization; concentration; minimizer; existence of optimal functions; regularity; uniqueness; Euler--Lagrange equations; free energy; nonlinear diffusion; mean field equations; nonlinear springs; measure valued solutions.}
\subjclass[2010]{Primary: 35A23; Secondary: 26D15, 35K55, 46E35, 49J40.}
\maketitle\thispagestyle{empty}

\section{Introduction}

We are concerned with the following minimization problem. For any $\lambda>0$ and any measurable function $\rho\ge0$ on $\R^N$, let
$$
I_\lambda[\rho] := \iint_{\R^N\times\R^N} |x-y|^\lambda\,\rho(x)\,\rho(y)\,dx\,dy\,.
$$
For $0<q<1$ we consider
$$
\mathcal C_{N,\lambda,q} := \inf\left\{ \frac{I_\lambda[\rho]}{\left( \int_{\R^N} \rho(x)\,dx \right)^\alpha\left( \int_{\R^N} \rho(x)^q\,dx \right)^{(2-\alpha)/q}} :\ 0\le\rho\in\mathrm L^1\cap\mathrm L^q(\R^N)\,,\ \rho\not\equiv 0 \right\}\,,
$$
where
$$
\alpha:=\frac{2\,N-q\,(2\,N+\lambda)}{N\,(1-q)}\,.
$$
By convention, for any $p>0$ we use the notation $\rho\in\mathrm L^p(\R^N)$ if $\irN{|\rho(x)|^p}$ is finite. Note that $\alpha$ is determined by scaling and homogeneity: for given values of $\lambda$ and $q$, the value of $\alpha$ is the only one for which there is a chance that the infimum is positive. We are asking whether $\mathcal C_{N,\lambda,q}$ is equal to zero or positive and, in the latter case, whether there is a unique minimizer. As we will see, there are three regimes $q<2\,N/(2\,N+\lambda)$, $q=2\,N/(2\,N+\lambda)$ and $q>2\,N/(2\,N+\lambda)$, which respectively correspond to $\alpha>0$, $\alpha=0$ and \hbox{$\alpha<0$}. The case $q=2\,N/(2\,N+\lambda)$, in which there is an additional \emph{conformal symmetry}, has already been dealt with in~\cite{DZ15} by J.~Dou and M.~Zhu, in~\cite[Theorem~18]{Beckner2015} by W.~Beckner, and in~\cite{MR3666824} by Q.A.~Ng\^o and V.H.~Nguyen, who have explicitly computed $\mathcal C_{N,\lambda,q}$ and characterized all solutions of the corresponding Euler--Lagrange equation. Here we will mostly concentrate on the other cases. Our main result is the following.
\begin{theorem}\label{main} Let $N\geq 1$, $\lambda>0$, $q\in(0,1)$ and define $\alpha$ as above. Then the inequality
\be{ineq:rHLS}
I_\lambda[\rho]\,\geq\mathcal C_{N,\lambda,q}\left( \int_{\R^N} \rho(x)\,dx \right)^\alpha\left( \int_{\R^N} \rho(x)^q\,dx \right)^{(2-\alpha)/q}
\ee
holds for any nonnegative function $\rho\in\mathrm L^1\cap\mathrm L^q(\R^N)$, for some positive constant $\mathcal C_{N,\lambda,q}$, if and only if $q>N/(N+\lambda)$. In this range, if either $N=1$, $2$ or if $N\ge3$ and $q\ge \min\big\{1-2/N\,,\,2\,N/(2\,N+\lambda)\big\}$, there is a radial positive, nonincreasing, bounded function $\rho\in\mathrm L^1\cap\mathrm L^q(\R^N)$ which achieves the equality case.\end{theorem}
This theorem provides a necessary and sufficient condition for the validity of the inequality, namely $q>N/(N+\lambda)$ or equivalently $\alpha<1$. Concerning the existence of an optimizer, the theorem completely answers this question in dimensions $N=1$ and $N=2$. In dimensions $N\geq 3$ we obtain a sufficient condition for the existence of an optimizer, namely, $q\geq \min\big\{1-2/N,2\,N/(2\,N+\lambda)\big\}$. This is not a necessary condition and, in fact, in Proposition~\ref{Cor:Mstar} we prove existence in a slightly larger, but less explicit region.

In the whole region $q>N/(N+\lambda)$ we are able to prove the existence of an optimizer for the \emph{relaxed inequality}
\be{ineq:rHLSrelaxed}
I_\lambda[\rho]+ 2M\Jlambda\rho x\,\geq\mathcal C_{N,\lambda,q}\left( \int_{\R^N} \rho(x)\,dx+ M \right)^\alpha\left( \int_{\R^N} \rho(x)^q\,dx \right)^{(2-\alpha)/q}
\ee
with the same optimal constant $\mathcal C_{N,\lambda,q}$. Here $\rho$ is an arbitrary nonnegative function in $\mathrm L^1\cap\mathrm L^q(\R^N)$ and $M$ an arbitrary nonnegative real number. If $M=0$, inequality~\eqref{ineq:rHLSrelaxed} is reduced to inequality~\eqref{ineq:rHLS}. It is straightforward to see that~\eqref{ineq:rHLSrelaxed} can be interpreted as the extension of~\eqref{ineq:rHLS} to measures with an absolutely continous part $\rho$ and an additional Dirac mass at the origin. Therefore the question about existence of an optimizer in Theorem~\ref{main} is reduced to the problem of whether the optimizer for this relaxed problem in fact has a Dirac mass. Fig.~\ref{Fig1} summarizes these considerations.

The optimizers have been explicitly characterized in the conformally invariant case $q=q(\lambda):=2\,N/(2\,N+\lambda)$ in~\cite{DZ15,Beckner2015,MR3666824} and are given, up to translations, dilations and multiplications by constants, by
$$
\rho(x)=\left(1+|x|^2\right)^{-N/q}\quad\forall\,x\in\R^N\,.
$$
This result determines the value of the optimal constant in~\eqref{ineq:rHLS} as
$$
\mathcal C_{N,\lambda,q(\lambda)}=\frac1{\pi^\frac\lambda2}\,\frac{\Gamma\left(\frac N2+\frac\lambda2\right)}{\Gamma\left(N+\frac\lambda2\right)}\,\left(\frac{\Gamma(N)}{\Gamma\left(\frac N2\right)}\right)^{1+\frac\lambda N}\,.
$$
By a simple argument that will be exposed in Section~\ref{Sec:Validity}, we can also find the optimizers in the special case $\lambda=2$: if $N/(N+2)<q<1$, then the optimizers for~\eqref{ineq:rHLS} are given by translations, dilations and constant multiples of
$$\rho(x)=\left(1+|x|^2\right)^{-\frac{1}{1-q}}\,.$$
In this case we obtain that
$$
\mathcal C_{N,2,q}=\frac{N\,(1-q)}{\pi\,q}\left(\frac{(N+2)\,q-N}{2\,q}\right)^{\frac{2-N\,(1-q)}{N\,(1-q)}}\left(\frac{\Gamma\left(\frac1{1-q}\right)}{\Gamma\left(\frac 1{1-q}-\frac N2\right)}\right)^\frac2N\,.
$$
Returning to the general case (that is, $q\neq2\,N/(2\,N+\lambda)$ and $\lambda\neq2$), no explicit form of the optimizers is known, but we can at least prove a uniqueness result in some cases, see also Fig.~\ref{Fig2}.
\begin{theorem}\label{prop:lambda=2to4}
Assume that $N/(N+\lambda)<q<1$ and either $q\ge1-1/N$ and $\lambda\geq 1$, or $2\leq\lambda\leq 4$. Then the optimizer for~\eqref{ineq:rHLSrelaxed} exists and is unique up to translation, dilation and multiplication by a positive constant.
\end{theorem}

\medskip We refer to~\eqref{ineq:rHLS} as a \emph{reverse Hardy--Littlewood--Sobolev inequality} as $\lambda$ is positive. The Hardy--Littlewood--Sobolev (HLS) inequality corresponds to negative values of $\lambda$ and is named after G.~Hardy and J.E.~Littlewood, see~\cite{MR1544927,MR1574995}, and S.L.~Sobolev, see~\cite{zbMATH03035784,sobolev1963theorem}; also see~\cite{MR944909} for an early discussion of rearrangement methods applied to these inequalities. In 1983, E.H.~Lieb in~\cite{MR717827} proved the existence of optimal functions for negative values of $\lambda$ and established optimal constants. His proof requires an analysis of the invariances which has been systematized under the name of \emph{competing symmetries}, see~\cite{MR1038450} and~\cite{MR1817225,burchard2009short} for a comprehensive introduction. Notice that rearrangement free proofs, which in some cases rely on the duality between Sobolev and HLS inequalities, have also been established more recently in various cases: see for instance~\cite{MR2659680,MR2925386,2014arXiv1404.1028J}. Standard HLS inequalities, which correspond to negative values of $\lambda$ in $I_\lambda[\rho]$, have many consequences in the theory of functional inequalities, particularly for identifying optimal constants.

Relatively few results are known in the case $\lambda>0$. The conformally invariant case, \emph{i.e.}, $q=2\,N/(2\,N+\lambda)$, appears in~\cite{DZ15} and is motivated by some earlier results on the sphere (see references therein). Further results have been obtained in~\cite{Beckner2015,MR3666824}, still in the conformally invariant case. Another range of exponents, which has no intersection with the one considered in the present paper, was studied earlier in~\cite[Theorem~G]{MR0121579}. Here we focus on a non-conformally invariant family of interpolation inequalities corresponding to a given $\mathrm L^1(\R^N)$ norm. In a sense, these inequalities play for HLS inequalities a role analogous to Gagliardo-Nirenberg inequalities compared to Sobolev's conformally invariant inequality.

\medskip The study of~\eqref{ineq:rHLS} is motivated by the analysis of nonnegative solutions to the evolution equation
\be{FP}
\partial_t\rho=\Delta\rho^q+\,\nabla\cdot\left(\rho\,\nabla W_\lambda\ast\rho\right)\,,
\ee
where the kernel is given by $W_\lambda(x):=\tfrac1\lambda\,|x|^\lambda$.
Eq.~\eqref{FP} is a special case of a larger family of Keller-Segel type equations, which covers the cases $q=1$ (linear diffusions), $q>1$ (diffusions of porous medium type) in addition to $0<q<1$ (fast diffusions), and also the range of exponents $\lambda<0$. Of particular interest is the original parabolic--elliptic Keller--Segel system which corresponds in dimension $N=2$ to a limit case as $\lambda\to0$, in which the kernel is $W_0(x)=\frac1{2\pi}\,\log|x|$ and the diffusion exponent is $q=1$. The reader is invited to refer to~\cite{hoffmann2017keller} for a global overview of this class of problems and for a detailed list of references and applications.

According to~\cite{AGS,santambrogio2015optimal},~\eqref{FP} has a gradient flow structure in the Wasserstein-2 metric. The corresponding \emph{free energy} functional is given by
$$
\mathcal F[\rho]:=-\frac{1}{1-q}\int_{\R^N}\rho^q\,dx+\frac{1}{2\lambda}\,I_\lambda[\rho]\quad\forall\,\rho\in\mathrm L^1_+(\R^N)\,,
$$
where $\mathrm L^1_+(\R^N)$ denotes the positive functions in $\mathrm L^1(\R^N)$.
As will be detailed later, optimal functions for~\eqref{ineq:rHLS} are energy minimizers for $\mathcal F$ under a \emph{mass} constraint. Smooth solutions $\rho(t,\cdot)$ of~\eqref{FP} with sufficient decay properties as $|x|\to+\,\infty$ conserve mass and center of mass over time while the free energy decays according to
$$
\frac d{dt}\mathcal F[\rho(t,\cdot)]=-\int_{\R^N}\rho\left|\tfrac q{1-q}\nabla\rho^{q-1}-\nabla W_\lambda\ast\rho\right|^2dx\,.
$$
This identity allows us to identify the smooth stationary solutions as the solutions of
$$
\rho_s=\left(C+(W_\lambda\ast\rho_s)\right)^{-\frac1{1-q}}
$$
where $C$ is a constant which has to be determined by the mass constraint. Thanks to the gradient flow structure, minimizers of the free energy $\F$ are stationary states of Eq.~\eqref{FP}. When dealing with solutions of~\eqref{FP} or with minimizers of the free energy, without loss of generality we can normalize the mass to $1$ in order to work in the space of probability measures $\P(\R^N)$. The general case of a bounded measure with an arbitrary mass can be recovered by an appropriate change of variables. Considering the \emph{lower semicontinuous extension of the free energy to $\P(\R^N)$} denoted by $\F^\Gamma$, we obtain counterparts to Theorems~\ref{main} and~\ref{prop:lambda=2to4} in terms of $\F^\Gamma$.
\begin{theorem}\label{FreeEnergy} The free energy $\F^\Gamma$ is bounded from below on $\P(\R^N)$ if and only if $N/(N+\lambda)<q<1$. If $q>N/(N+\lambda)$, then there exists a global minimizer $\mu_*\in\mathcal{P}(\R^N)$ and, modulo translations, it has the form
\begin{equation*}
\mu_*=\rho_*+M_*\,\delta_0
\end{equation*}
for some $M_*\in[0,1)$. Moreover $\rho_*\in\mathrm L^1_+\cap \mathrm L^q(\R^N)$ is radially symmetric, non-increasing and supported on $\R^N$.

If $M_*=0$, then $\rho_*$ is an optimizer of~\eqref{ineq:rHLS}. Conversely, if $\rho\in\mathrm L^1_+\cap\mathrm L^q(\R^N)$ is an optimizer of~\eqref{ineq:rHLS} with mass $M>0$, then there is an explicit $\ell>0$ such that $\ell^{-N}\rho(x/\ell)/M$ is a global minimizer of $\F^\Gamma$ on $\P(\R^N)$.

Finally, if $N/(N+\lambda)<q<1$ and either $q\ge1-1/N$ and $\lambda\geq 1$, or $2\leq\lambda\leq 4$, then the global minimizer $\mu_*$ of $\F^\Gamma$ on $\mathcal P(\R^N)$ is \emph{unique} up to translation.
\end{theorem}
In the region of the parameters of Theorem~\ref{main} for which~\eqref{ineq:rHLS} is achieved by a radial function, this optimizer is also a minimizer of $\mathcal F$. If the minimizer $\mu_*$ of $\F^\Gamma$ has a singular part, then the constant $\mathcal C_{N,\lambda,q}$ is also achieved by $\mu_*$ in~\eqref{ineq:rHLSrelaxed}, up to a translation. Hence the results of Theorem~\ref{FreeEnergy} are equivalent to the results of Theorems~\ref{main} and~\ref{prop:lambda=2to4}.

The use of free energies to understand the long-time asymptotics of gradient flow equations like~\eqref{FP} and various related models with other interaction potentials than $W_\lambda$ or more general pressure variables than $\rho^{q-1}$ has already been studied in some cases: see for instance \cite{AGS,CaMcCVi03,CaMcCVi06,Villani03}. The connection to Hardy--Littlewood--Sobolev type functional inequalities~\cite{CCL,BCL,CCH1} is well-known for the range $\lambda\in (-N,0]$. However, the case of $W_\lambda$ with $\lambda>0$ is as far as we know entirely new.

\medskip This paper results from the merging of two earlier preprints,~\cite{2018arXiv180306151D} and~\cite{2018arXiv180306232C}, corresponding to two research projects that were investigated independently.

\medskip Section~\ref{Sec:Validity} is devoted to the proof of the reverse HLS inequality~\eqref{ineq:rHLS} and also of the optimal constant in the case $\lambda=2$. In Section~\ref{Sec:Existence} we study the existence of optimizers of the reverse HLS inequality via the relaxed variational problem associated with~\eqref{ineq:rHLSrelaxed}. The regularity properties of these optimizers are analysed in Section~\ref{Sec:AdditionalResults}, with the goal of providing some additional results of no-concentration. Section~\ref{sec:freeenergy} is devoted to the equivalence of the reverse HLS inequalities and the existence of a lower bound of $\F^\Gamma$ on $\P(\R^N)$. The relative compactness of minimizing sequences of probability measures is also established as well as the uniqueness of the measure valued minimizers of $\F^\Gamma$, in the same range of the parameters as in Theorem~\ref{prop:lambda=2to4}. We conclude this paper by an Appendix~\ref{Appendix} on a toy model for concentration which sheds some light on the threshold value $q=1-2/N$ and by another Appendix~\ref{Sec:qge1} devoted to the simpler case $q\ge1$, in order to complete the picture. From here on (except in Appendix~\ref{Sec:qge1}), we shall assume that $q<1$ without further notice.

\section{Reverse HLS inequality}\label{Sec:Validity}

The following proposition gives a necessary and sufficient condition for inequality~\eqref{ineq:rHLS}.
\begin{proposition}\label{ineq}
Let $\lambda>0$.
\begin{enumerate}
\item If $0<q\le N/(N+\lambda)$, then $\mathcal C_{N,\lambda,q}=0$.
\item If $N/(N+\lambda)<q<1$, then $\mathcal C_{N,\lambda,q}>0$.
\end{enumerate}
\end{proposition}
The result for $q<N/(N+\lambda)$ was obtained in~\cite{CDP} using a different method. The result for $q=N/(N+\lambda)$, as well as the result for $2\,N/(2\,N+\lambda)\neq q>N/(N+\lambda)$, are new.
\begin{proof}[Proof of Proposition~\ref{ineq}. Part (1).]
Let $\rho\ge0$ be bounded with compact support and let $\sigma\ge0$ be a smooth function with $\int_{\R^N}\sigma(x)\,dx =1$. With another parameter $M>0$ we consider
$$
\rho_\epsilon(x) = \rho(x)+ M\,\epsilon^{-N}\,\sigma(x/\epsilon)\,,
$$
where $\epsilon>0$ is a small parameter. Then $\int_{\R^N} \rho_\epsilon(x)\,dx = \int_{\R^N} \rho(x)\,dx+ M$ and, by simple estimates,
\begin{equation}
\label{eq:limitq}
\int_{\R^N} \rho_\epsilon(x)^q\,dx \to \int_{\R^N} \rho(x)^q\,dx
\quad\text{as}\quad\epsilon\to 0_+
\end{equation}
and
$$
I_\lambda[\rho_\epsilon] \to I_\lambda[\rho]+ 2M\Jlambda\rho x
\quad\text{as}\quad\epsilon\to 0_+\,.
$$
Thus, taking $\rho_\epsilon$ as a trial function,
\begin{equation}
\label{eq:functionalplusdelta}
\mathcal C_{N,\lambda,q}\le\frac{I_\lambda[\rho]+ 2M\Jlambda\rho x}{\left( \int_{\R^N} \rho(x)\,dx+ M \right)^\alpha\left( \int_{\R^N} \rho(x)^q\,dx \right)^{(2-\alpha)/q} }=:\mathcal Q[\rho,M]\,.
\end{equation}
This inequality is valid for any $M$ and therefore we can let $M\to+\,\infty$. If $\alpha>1$, which is the same as $q<N/(N+\lambda)$, we immediately obtain $\mathcal C_{N,\lambda,q}=0$ by letting $M\to+\,\infty$. If $\alpha=1$, \emph{i.e.}, $q=N/(N+\lambda)$, by taking the limit as $M\to+\,\infty$, we obtain
$$
\mathcal C_{N,\lambda,q}\le\frac{2\Jlambda\rho x}{\left( \int_{\R^N} \rho(x)^q\,dx \right)^{1/q} }\,.
$$
Let us show that by a suitable choice of $\rho$ the right side can be made arbitrarily small. For any $R>1$, we take
$$
\rho_R(x):=|x|^{-(N+\lambda)}\,\mathbbm 1_{1\le|x|\le R}(x)\,.
$$
Then
$$
\int_{\R^N} |x|^\lambda\,\rho_R\,dx=\int_{\R^N}\rho_R^q\,dx=\left|\Sph^{N-1}\right|\,\log R
$$
and, as a consequence,
$$
\frac{\Jlambda{\rho_R}x}{\left( \int_{\R^N} \rho_R^{N/(N+\lambda)}\,dx \right)^{(N+\lambda)/N}}=\Big(\left|\Sph^{N-1}\right|\,\log R\Big)^{-\lambda/N} \to 0\quad\text{as}\quad R\to\infty\,.
$$
This proves that $\mathcal C_{N,\lambda,q}=0$ for $q=N/(N+\lambda)$.
\end{proof}

In order to prove that $\mathcal C_{N,\lambda,q}>0$ in the remaining cases, we need the following simple bound, which is known as a \emph{Carlson type inequality} in the literature after \cite{zbMATH03014034} and whose sharp form has been established in~\cite{levin1948exact} by V.~Levin. Various proofs can be found in the literature and we insist on the fact that they are not limited to the case $q<1$: see for instance~\cite[Ineq.~2(a)]{MR8416}, \cite[Chap.~VII, Ineq. (8.1)]{MR1190927} or~\cite[Section~4]{MR2039942}. For completeness, we give a statement and a proof for the case we are interested in.
\begin{lemma}[Carlson-Levin inequality]\label{ineq2}
Let $\lambda>0$ and $N/(N+\lambda)<q<1$. Then there is a constant $c_{N,\lambda,q}>0$ such that for all $\rho\ge0$,
$$
\left( \int_{\R^N} \rho\,dx \right)^{1-\frac{N\,(1-q)}{\lambda\,q}}\left(\Jlambda\rho x\right)^\frac{N\,(1-q)}{\lambda\,q}\ge c_{N,\lambda,q}\left(\int_{\R^N}\rho^q\,dx\right)^{1/q}\,.
$$
Equality is achieved if and only if
$$
\rho(x)=\left(1+|x|^\lambda\right)^{-\frac{1}{1-q}}
$$
up to translations, dilations and constant multiples, and one has
$$c_{N,\lambda,q}=\frac1\lambda\left(\frac{(N+\lambda)\,q-N}q\right)^\frac1q\left(\frac{N\,(1-q)}{(N+\lambda)\,q-N}\right)^{\frac N\lambda\,\frac{1-q}q}
\left(\frac{\Gamma\left(\frac N2\right)\,\Gamma\left(\frac1{1-q}\right)}{2\,\pi^\frac N2\,\Gamma\left(\frac 1{1-q}-\frac N\lambda\right)\,\Gamma\left(\frac N\lambda\right)}\right)^\frac{1-q}q\,.$$
\end{lemma}
\begin{proof}
Let $R>0$. Using H\"older's inequality in two different ways, we obtain
$$
\int_{\{|x|<R\}} \rho^q\,dx\le\left( \int_{\R^N} \rho\,dx \right)^q |B_R|^{1-q} = C_1\left( \int_{\R^N} \rho\,dx \right)^q R^{N\,(1-q)}
$$
and
\begin{align*}
\int_{\{|x|\ge R\}} \rho^q\,dx
& \le\left(\Jlambda\rho x\right)^q\left( \int_{\{|x|\ge R\}} |x|^{-\frac{\lambda\,q}{1-q}}\,dx \right)^{1-q} \\
& = C_2\left(\Jlambda\rho x\right)^q R^{-\lambda\,q+ N\,(1-q)}\,.
\end{align*}
The fact that $C_2<\infty$ comes from the assumption $q>N/(N+\lambda)$, which is the same as $\lambda\,q/(1-q)>N$. To conclude, we add these two inequalities and optimize over $R$.

The existence of a radial monotone non-increasing optimal function follows by standard variational methods; the expression for the optimal functions is a consequence of the Euler-Lagrange equations. The expression of $c_{N,\lambda,q}$ is then straightforward.
\end{proof}

\begin{proof}[Proof of Proposition~\ref{ineq}. Part (2).]
By rearrangement inequalities it suffices to prove the inequality for symmetric non-increasing $\rho$'s. For such functions, by the simplest rearrangement inequality,
\begin{equation*}
\int_{\R^N} |x-y|^\lambda \rho(y)\,dx\ge\Jlambda\rho x\quad\text{for all}\quad x\in\R^N\,.
\end{equation*}
Thus,
\begin{equation}
\label{eq:lowerboundsymmdecr}
I_\lambda[\rho]\ge\Jlambda\rho x\int_{\R^N}\rho\,dx\,.
\end{equation}
In the range $\frac N{N+\lambda}<q<1$ (for which $\alpha<1$), we recall that by Lemma~\ref{ineq2}, for any symmetric non-increasing function $\rho$, we have
$$
\frac{I_\lambda[\rho]}{\left(\irN{\rho(x)}\right)^\alpha}\ge\left(\irN\rho\,dx\right)^{1-\alpha}\Jlambda\rho x\ge c_{N,\lambda,q}^{2-\alpha}\left(\int_{\R^N}\rho^q\,dx\right)^\frac{2-\alpha}q
$$
because $2-\alpha=\frac{\lambda\,q}{N\,(1-q)}$. As a consequence, we obtain that
$$
\mathcal C_{N,\lambda,q}\ge c_{N,\lambda,q}^{2-\alpha}>0\,.
$$
\end{proof}

\begin{corollary}\label{cor:lambda=2}
Let $\lambda=2$ and $N/(N+2)<q<1$. Then the optimizers for~\eqref{ineq:rHLS} are given by translations, dilations and constant multiples of
$$\rho(x)=\left(1+|x|^2\right)^{-\frac{1}{1-q}}$$
and the optimal constant is
$$
\mathcal C_{N,2,q}=2\,c_{N,2,q}^{\frac{2\,q}{N\,(1-q)}}\,.
$$
\end{corollary}
\begin{proof} By rearrangement inequalities it is enough to prove~\eqref{main} for symmetric non-increa\-sing $\rho$'s, and so $\int_{\R^N} x\,\rho(x)\,dx=0$. Therefore
\begin{equation*}\label{eq:lambda2com0}
I_2[\rho] = 2 \int_{\R^N} \rho(x)\,dx \int_{\R^N} |x|^2\,\rho(x)\,dx
\end{equation*}
and the optimal function is the one of the Carlson type inequality of Lemma~\ref{ineq2}.\end{proof}
By taking into account the fact that
$$
c_{N,2,q}=\frac12\left(\frac{(N+2)\,q-N}q\right)^\frac1q\left(\frac{N\,(1-q)}{(N+2)\,q-N}\right)^{\frac{N}{2}\frac{1-q}{q}}\,\left(\frac{\Gamma\left(\frac1{1-q}\right)}{2\,\pi^\frac N2\,\Gamma\left(\frac 1{1-q}-\frac N2\right)}\right)^\frac{1-q}q\,,
$$
we recover the expression of $\mathcal C_{N,2,q}$ given in the introduction.

\begin{remark} We can now make a few observations on the reverse HLS inequality~\eqref{ineq:rHLS} and its optimal constant $\mathcal C_{N,\lambda,q}$.
\\[4pt]
(i) The computation in the proof of Proposition~\ref{ineq}, Part (2) explains a surprising feature of~\eqref{ineq:rHLS}: $I_\lambda[\rho]$ controls a product of two terms. However, in the range $N/(N+\lambda)<q<2\,N/(2\,N+\lambda)$ which corresponds to $\alpha\in(0,1)$, the problem is actually reduced (with a non-optimal constant) to the interpolation of $\irN{\rho^q}$ between $\irN\rho$ and $\Jlambda\rho x$, which has a more classical structure.
\\[4pt]
(ii) There is an alternative way to prove~\eqref{ineq:rHLS} in the range $2\,N/(2\,N+\lambda)<q<1$ using the results from~\cite{DZ15,Beckner2015,MR3666824}. We can indeed rely on H\"older's inequality to get that
$$
\left(\irN{\rho(x)^q}\right)^{1/q}\le\left(\irN{\rho(x)^\frac{2\,N}{2\,N+\lambda}}\right)^{\eta\,\frac{2\,N+\lambda}{2\,N}}\left(\irN\rho\right)^{1-\eta}
$$
with $\eta:=\frac{2\,N\,(1-q)}{\lambda\,q}$. By applying the conformally invariant inequality
$$
I_\lambda[\rho]\ge\mathcal C_{N,\lambda,\frac{2\,N}{2\,N+\lambda}}\left(\irN{\rho(x)^\frac{2\,N}{2\,N+\lambda}}\right)^\frac{2\,N+\lambda}N
$$
shown in~\cite{DZ15,Beckner2015,MR3666824}, we obtain that
$$
\mathcal C_{N,\lambda,q}\ge\mathcal C_{N,\lambda,\frac{2\,N}{2\,N+\lambda}}=\pi^{-\frac\lambda2}\,\frac{\Gamma\big(\frac N2+\frac\lambda2\big)}{\Gamma\big(N+\frac\lambda2\big)}\left(\frac{\Gamma(N)}{\Gamma\big(\frac N2\big)}\right)^{1+\frac\lambda N}\,.
$$
We notice that $\alpha=-\,2\,(1-\eta)/\eta$ is negative in the range $2\,N/(2\,N+\lambda)<q<1$.
\\[4pt]
(iii) We have
$$
\lim_{q\to N/(N+\lambda)_+}\mathcal C_{N,\lambda,q}=0
$$
because the map $(\lambda,q)\mapsto\mathcal C_{N,\lambda,q}$ is \emph{upper semi-continuous}. The proof of this last property goes as follows. Let us rewrite $\mathcal Q[\rho,0]$ defined in~\eqref{eq:functionalplusdelta} as
\begin{equation}\label{newQ}
\mathsf Q_{q,\lambda}[\rho]:=\frac{I_\lambda[\rho]}{\left( \int_{\R^N} \rho(x)\,dx \right)^\alpha\left( \int_{\R^N} \rho(x)^q\,dx \right)^{(2-\alpha)/q}}\,.
\end{equation}
In this expression of the energy quotient, we emphasize the dependence in $q$ and $\lambda$. As before, the infimum of $\mathsf Q_{q,\lambda}$ over the set of nonnegative functions in $\mathrm L^1\cap\mathrm L^q(\R^N)$ is $\mathcal C_{N,\lambda,q}$. Let $(q,\lambda)$ be a given point in $(0,1)\times(0,\infty)$ and let $(q_n,\lambda_n)_{n\in\N}$ be a sequence converging to $(q,\lambda)$. Let $\epsilon>0$ and choose a $\rho$ which is bounded, has compact support and is such that $\mathsf Q_{q,\lambda}[\rho]\leq \mathcal C_{N,\lambda,q}+ \epsilon$. Then, by the definition as an infimum, $\mathcal C_{N,q_n,\lambda_n} \leq\mathsf Q_{q_n,\lambda_n}[\rho]$. On the other hand, the assumptions on~$\rho$ imply that $\lim_{n\to\infty}\mathsf Q_{q_n,\lambda_n}[\rho]=\mathsf Q_{q,\lambda}[\rho]$. We conclude that $\limsup_{n\to\infty}\mathcal C_{N,q_n,\lambda_n}\leq\mathcal C_{N,\lambda,q}+ \epsilon$. Since $\epsilon$ is arbitrary, we obtain the claimed upper semi-continuity property.
\end{remark}

\section{Existence of minimizers and relaxation}\label{Sec:Existence}

We now investigate whether there are nonnegative minimizers in $\mathrm L^1\cap\mathrm L^q(\R^N)$ for $\mathcal C_{N,\lambda,q}$ if $N/(N+\lambda)<q<1$. As mentioned before, the conformally invariant case $q=2\,N/(2\,N+\lambda)$ has been dealt with before and will be excluded from our considerations. We start with the simpler case $2\,N/(2\,N+\lambda)<q<1$, which corresponds to $\alpha<0$.
\begin{proposition}\label{opt0}
Let $\lambda>0$ and $2\,N/(2\,N+\lambda)<q<1$. Then there is a minimizer for $\mathcal C_{N,\lambda,q}$.
\end{proposition}
\begin{proof}
Let $(\rho_j)_{j\in\mathbb N}$ be a minimizing sequence. By rearrangement inequalities we may assume that the $\rho_j$ are symmetric non-increasing. By scaling and homogeneity, we may also assume that
$$
\int_{\R^N} \rho_j(x)\,dx = \int_{\R^N}\rho_j(x)^q\,dx = 1
\quad\text{for all}\,j\in\mathbb N\,.
$$
This together with the symmetric non-increasing character of $\rho_j$ implies that
$$
\rho_j(x)\le C\,\min\left\{ |x|^{-N},\,|x|^{-N/q}\right\}
$$
with $C$ independent of $j$. By Helly's selection theorem we may assume, after passing to a subsequence if necessary, that $\rho_j\to\rho$ almost everywhere. The function $\rho$ is symmetric non-increasing and satisfies the same upper bound as $\rho_j$.

By Fatou's lemma we have
$$
\liminf_{j\to\infty} I_\lambda[\rho_j]\ge I_\lambda[\rho]
\quad\text{and}\quad
1\ge \int_{\R^N} \rho(x)\,dx\,.
$$
To complete the proof we need to show that $\int_{\R^N} \rho(x)^q\,dx =1$ (which implies, in particular, that $\rho\not\equiv 0$) and then $\rho$ will be an optimizer.

Modifying an idea from~\cite{MR3273640} we pick $p\in\left(N/(N+\lambda),q\right)$ and apply~\eqref{ineq:rHLS} with the same~$\lambda$ and $\alpha(p)=\big(2\,N-p\,(2\,N+\lambda)\big)\big/\big(N\,(1-p)\big)$ to get
$$
I_\lambda[\rho_j]\ge\mathcal C_{N,\lambda,p}\left( \int_{\R^N} \rho_j^{p}\,dx \right)^{(2-\alpha(p))/p}\,.
$$
Since the left side converges to a finite limit, namely $\mathcal C_{N,\lambda,q}$, we find that the $\rho_j$ are uniformly bounded in $\mathrm L^{p}(\R^N)$ and therefore we have as before
$$
\rho_j(x)\le C'\,|x|^{-N/p}\,.
$$
Since $\min\left\{|x|^{-N},|x|^{-N/p}\right\}\in\mathrm L^q(\R^N)$, we obtain by dominated convergence
$$
\int_{\R^N} \rho_j^q\,dx \to \int_{\R^N} \rho^q\,dx\,,
$$
which, in view of the normalization, implies that $\int_{\R^N} \rho(x)^q\,dx =1$, as claimed.
\end{proof}

Next, we prove the existence of minimizers in the range $N/(N+\lambda)<q<2\,N/(2\,N+\lambda)$ by considering the \emph{minimization of the relaxed problem}~\eqref{ineq:rHLSrelaxed}. The idea behind this relaxation is to allow $\rho$ to contain a Dirac function at the origin. The motivation comes from the proof of the first part of Proposition~\ref{ineq}. The expression of $\mathcal Q[\rho,M]$ as defined in~\eqref{eq:functionalplusdelta} arises precisely from a measurable function $\rho$ together with a Dirac function of strength $M$ at the origin. We have seen that in the regime $q\le N/(N+\lambda)$ (that is, $\alpha\ge1$) it is advantageous to increase~$M$ to infinity. This is no longer so if $N/(N+\lambda)<q<2\,N/(2\,N+\lambda)$. While it is certainly disadvantageous to move $M$ to infinity, it has to be investigated whether the optimum $M$ is $0$ or a positive finite value.

Let
$$
\mathcal C_{N,\lambda,q}^{\rm{rel}} := \inf\Big\{\mathcal Q[\rho,M]\,:\,0\le\rho\in\mathrm L^1\cap\mathrm L^q(\R^N)\,,\ \rho\not\equiv 0\,,\ M\ge0 \Big\}
$$
where $\mathcal Q[\rho,M]$ is defined by~\eqref{eq:functionalplusdelta}. We know that $\mathcal C_{N,\lambda,q}^{\rm{rel}}\le\mathcal C_{N,\lambda,q}$ by restricting the minimization to $M=0$. On the other hand,~\eqref{eq:functionalplusdelta} gives $\mathcal C_{N,\lambda,q}^{\rm{rel}}\ge\mathcal C_{N,\lambda,q}$. Therefore,
$$
\mathcal C_{N,\lambda,q}^{\rm{rel}}= \mathcal C_{N,\lambda,q}\,,
$$
which justifies our interpretation of $\mathcal C_{N,\lambda,q}^{\rm{rel}}$ as a \emph{relaxed minimization} problem. Let us start with a preliminary observation.
\begin{lemma}\label{sym-pos} Let $\lambda>0$ and $N/(N+\lambda)<q<1$. If $\rho\ge0$ is an optimal function for either $\mathcal C_{N,\lambda,q}^{\rm{rel}}$ (for an $M\geq 0$) or $\mathcal C_{N,\lambda,q}$ (with $M=0$), then $\rho$ is radial (up to a translation), monotone non-increasing and positive almost everywhere on $\R^N$.\end{lemma}
\begin{proof} Since $\mathcal C_{N,\lambda,q}$ is positive, we observe that $\rho$ is not identically $0$. By rearrangement inequalities and up to a translation, we know that~$\rho$ is radial and monotone non-increa\-sing. Assume by contradiction that $\rho$ vanishes on a set $E\subset\R^N$ of finite, positive measure. Then
$$
\mathcal Q\big[\rho,M+\epsilon\,\1_E\big]=\mathcal Q[\rho,M]\left(1-\frac{2-\alpha}q\,\frac{|E|}{\irN{\rho(x)^q}}\,\epsilon^q+o(\epsilon^q)\right)
$$
as $\varepsilon\to0_+$, a contradiction to the minimality for sufficiently small $\epsilon>0$.
\end{proof}
Varying $\mathcal Q[\rho,M]$ with respect to $\rho$, we obtain \emph{the Euler--Lagrange equation} on $\R^N$ for any minimizer $(\rho_*,M_*)$ for $\mathcal C_{N,\lambda,q}^{\rm{rel}}$:
\begin{equation}\label{EL}
2\,\frac{\int_{\R^N} |x-y|^\lambda \rho_*(y)\,dy+ M_* |x|^\lambda}{I_\lambda[\rho_*]+ 2M_*\Jlambda{\rho_*}y} - \frac\alpha{\int_{\R^N} \rho_*\,dy+M_*} - (2-\alpha)\,\frac{\rho_*(x)^{-1+q}}{\int_{\R^N} \rho_*(y)^q\,dy} = 0\,.
\end{equation}
This equation follows from the fact that $\rho_*$ is positive almost everywhere according to Lemma~\ref{sym-pos}.
\begin{proposition}\label{opt1}
Let $\lambda>0$ and $N/(N+\lambda)<q<2\,N/(2\,N+\lambda)$. Then there is a minimizer for $\mathcal C_{N,\lambda,q}^{\rm{rel}}$.
\end{proposition}
We will later show that for $N=1$ and $N=2$ there is a minimizer for the original problem $\mathcal C_{N,\lambda,q}$ in the full range of $\lambda$'s and $q$'s covered by Proposition~\ref{opt1}. If $N\ge3$, the same is true under additional restrictions.

\begin{proof}[Proof of Proposition~\ref{opt1}]
The beginning of the proof is similar to that of Proposition~\ref{opt0}. Let $(\rho_j,M_j)_{j\in\N}$ be a minimizing sequence. By rearrangement inequalities we may assume that $\rho_j$ is symmetric non-increasing. Moreover, by scaling and homogeneity, we may assume that
$$
\int_{\R^N} \rho_j\,dx+M_j = \int_{\R^N} \rho_j^q = 1\,.
$$
In a standard way this implies that
$$
\rho_j(x)\le C\,\min\left\{ |x|^{-N}, |x|^{-N/q}\right\}
$$
with $C$ independent of $j$. By Helly's selection theorem we may assume, after passing to a subsequence if necessary, that $\rho_j\to\rho$ almost everywhere. The function $\rho$ is symmetric non-increasing and satisfies the same upper bound as $\rho_j$. Passing to a further subsequence, we can also assume that $(M_j)_{j\in\N}$ and $\left(\irN{\rho_j}\right)_{j\in\N}$ converge and define $M:=L+\lim_{j\to\infty}M_j$ where $L=\lim_{j\to\infty}\irN{\rho_j}-\irN\rho$, so that $\irN\rho+M=1$. In the same way as before, we show that
$$
\int_{\R^N} \rho(x)^q\,dx = 1\,.
$$

We now turn our attention to the $\mathrm L^1$-term. We cannot invoke Fatou's lemma because $\alpha\in(0,1)$ and therefore this term appears in $\mathcal Q$ with a positive exponent in the denominator. The problem with this term is that $|x|^{-N}$ is not integrable at the origin and we cannot get a better bound there. We have to argue via measures, so let $d\mu_j(x) := \rho_j(x)\,dx$. Because of the upper bound on $\rho_j$ we have
$$
\mu_j\left(\R^N\setminus B_R(0)\right) = \int_{\{|x|\ge R\}} \rho_j(x)\,dx\le C \int_{\{|x|\ge R\}} \frac{dx}{|x|^{N/q}} = C'\,R^{-N\,(1-q)/q}\,.
$$
This means that the measures are tight. After passing to a subsequence if necessary, we may assume that $\mu_j\to\mu$ weak * in the space of measures on $\R^N$. Tightness implies that
$$
\mu(\R^N)=\lim_{j\to\infty} \int_{\R^N} \rho_j\,dx = L+\irN\rho\,.
$$
Moreover, since the bound $C\,|x|^{-N/q}$ is integrable away from any neighborhood of the origin, we see that $\mu$ is absolutely continuous on $\R^N\setminus\{0\}$ and $d\mu/dx =\rho$. In other words,
$$
d\mu = \rho\,dx+L\,\delta\,.
$$

Using weak convergence in the space of measures one can show that
$$
\liminf_{j\to\infty} I_\lambda[\rho_j]\ge I_\lambda[\rho]+ 2M\Jlambda\rho x\,.
$$
Finally, by Fatou's lemma,
$$
\liminf_{j\to\infty}\Jlambda{\rho_j}x\ge\int_{\R^N} |x|^\lambda\left(\rho(x)\,dx+L\,\delta\right) =\Jlambda\rho x\,.
$$
Hence
$$
\liminf_{j\to\infty}\mathcal Q[\rho_j,M_j]\ge\mathcal Q[\rho,M]\,.
$$
By definition of $\mathcal C_{N,\lambda,q}^{\rm{rel}}$ the right side is bounded from below by $\mathcal C_{N,\lambda,q}^{\rm{rel}}$. On the other hand, by choice of $\rho_j$ and $M_j$ the left side is equal to $\mathcal C_{N,\lambda,q}^{\rm{rel}}$. This proves that $(\rho, M)$ is a minimizer for $\mathcal C_{N,\lambda,q}^{\rm{rel}}$.
\end{proof}

Next, we show that under certain assumptions a minimizer $(\rho_*,M_*)$ for the relaxed problem must, in fact, have $M_*=0$ and is therefore a minimizer of the original problem.
\begin{proposition}\label{opt2}
Let $N\geq 1 $, $\lambda>0$ and $N/(N+\lambda)<q<2\,N/(2\,N+\lambda)$. If $N\ge3$ and
$\lambda >2\,N/(N-2)$, then assume in addition that $q\ge1-2/N $. If $(\rho_*,M_*)$ is a minimizer for $\mathcal C_{N,\lambda,q}^{\rm{rel}}$, then $M_*=0$. In particular, there is a minimizer for $\mathcal C_{N,\lambda,q}$.
\end{proposition}
Note that for $N\ge3$, we are implicitly assuming $\lambda<4N/(N-2)$ since otherwise the two assumptions $q<2\,N/(2\,N+\lambda)$ and $q\ge1-2/N$ cannot be simultaneously satisfied. For the proof of Proposition~\ref{opt2} we need the following lemma which identifies the sub-leading term in~\eqref{eq:limitq}.
\begin{lemma}\label{bl}
Let $0<q<p$, let $f\in\mathrm L^p\cap\mathrm L^q(\R^N)$ be a symmetric non-increasing function and let $g\in\mathrm L^q(\R^N)$. Then, for any $\tau>0$, as $\epsilon\to 0_+$,
$$
\int_{\R^N}\left|f(x)+ \epsilon^{-N/p}\,\tau\,g(x/\epsilon)\right|^q\,dx = \int_{\R^N} f^q\,dx+ \epsilon^{N(1-q/p)}\,\tau^q\int_{\R^N} |g|^q\,dx+ o\left(\epsilon^{N(1-q/p)}\,\tau^q\right)\,.
$$
\end{lemma}
\begin{proof}[Proof of Lemma~\ref{bl}]
We first note that
\begin{equation}
\label{eq:fsingularity}
f(x) = o\left(|x|^{-N/p}\right)
\quad\text{as}\quad x\to 0
\end{equation}
in the sense that for any $c>0$ there is an $r>0$ such that for all $x\in\R^N$ with $|x|\le r $ one has $f(x)\le c\,|x|^{-N/p}$. To see this, we note that, since $f$ is symmetric non-increasing,
$$
f(x)^p\le\frac{1}{\left|\left\{ y\in\R^N:\ |y|\le|x|\right\}\right|} \int_{|y|\le|x|} f(y)^p\,dy\,.
$$
The bound~\eqref{eq:fsingularity} now follows by dominated convergence.

It follows from~\eqref{eq:fsingularity} that, as $\epsilon\to 0_+$,
$$
\epsilon^{N/p} f(\epsilon x) \to 0
\quad\text{for any}\quad x\in\R^N\,,
$$
and therefore, in particular, $\tau\,g(x)+ \epsilon^{N/p} f(\epsilon x)\to \tau\,g(x)$ for any $x\in\R^N$. From the Br\'ezis--Lieb lemma (see~\cite{MR699419}) we know that
$$
\int_{\R^N} \left|\tau\,g(x)+ \epsilon^{N/p} f(\epsilon x)\right|^q\,dx = \tau^q\int_{\R^N} |g(x)|^q\,dx+ \int_{\R^N}\left( \epsilon^{N/p} f(\epsilon x)\right)^q\,dx+ o(1)\,.
$$
By scaling this is equivalent to the assertion of the lemma.
\end{proof}

\begin{proof}[Proof of Proposition~\ref{opt2}]
We argue by contradiction and assume that $M_*>0$. Let $0\le\sigma\in\left(\mathrm L^1\cap\mathrm L^q\left(\R^N\right)\right)\cap\mathrm L^1\left(\R^N,|x|^\lambda\,dx\right)$ with $\int_{\R^N} \sigma\,dx =1$. We compute the value of $\mathcal Q[\rho,M]$ for the family $(\rho,M)=\left(\rho_*+ \epsilon^{-N} \tau\,\sigma(\cdot/\epsilon),M_*-\tau\right)$ with a parameter $\tau<M_*$.

\noindent 1) We have
\begin{multline*}
I_\lambda\left[\rho_*+ \epsilon^{-N} \tau\,\sigma(\cdot/\epsilon)\right]+ 2\,(M_*-\tau) \int_{\R^N} |x|^\lambda\left(\rho_*(x)+ \epsilon^{-N} \tau\,\sigma(x/\epsilon)\right) dx \\
= I_\lambda[\rho_*]+ 2\,M_*\Jlambda{\rho_*}x+ R_1
\end{multline*}
with
\begin{multline*}
R_1 = 2\,\tau \iint_{\R^N\times\R^N} \rho_*(x)\left( |x-y|^\lambda - |x|^\lambda\right) \epsilon^{-N} \sigma(y/\epsilon)\,dx\,dy \\
+ \epsilon^\lambda\,\tau^2\,I_\lambda[\sigma]+ 2\,(M_*-\tau)\,\tau\,\epsilon^\lambda\Jlambda\sigma x\,.
\end{multline*}
Let us show that $R_1=O\left(\epsilon^\beta\,\tau\right)$ with $\beta:=\min\{2,\lambda\}$. This is clear for the last two terms in the definition of $R_1$, so it remains to consider the double integral. If $\lambda\le1$ we use the simple inequality $|x-y|^\lambda-|x|^\lambda\le|y|^\lambda$ to conclude that
$$
\iint_{\R^N\times\R^N} \rho_*(x)\left( |x-y|^\lambda - |x|^\lambda\right) \epsilon^{-N} \sigma(y/\epsilon)\,dx\,dy\le\epsilon^\lambda\Jlambda\sigma x\int_{\R^N}\rho_*\,dx\,.
$$
If $\lambda>1$ we use the fact that, with a constant $C$ depending only on $\lambda$,
\begin{equation}
\label{eq:elemineq}
|x-y|^\lambda - |x|^\lambda\le -\lambda |x|^{\lambda-2} x\cdot y+ C\left(|x|^{ (2-\lambda)_+ } |y|^\beta+ |y|^\lambda \right)\,.
\end{equation}
Since $\rho_*$ is radial, we obtain
\begin{multline*}
\iint_{\R^N\times\R^N} \rho_*(x)\left( |x-y|^\lambda - |x|^\lambda\right) \epsilon^{-N} \sigma(y/\epsilon)\,dx\,dy \\
\le C\left( \epsilon^\beta \int_{\R^N} |x|^{(2-\lambda)_+} \rho_*(x)\,dx \int_{\R^N} |y|^\beta \sigma(y)\,dy+ \epsilon^\lambda\Jlambda\sigma x\int_{\R^N}\rho_*(x)\,dx\right)\,.
\end{multline*}
Using the fact that $\rho_*$, $\sigma\in\mathrm L^1\left(\R^N\right)\cap\mathrm L^1\left(\R^N,|x|^\lambda\,dx\right)$ it is easy to see that the integrals on the right side are finite, so indeed $R_1=O\left(\epsilon^\beta\,\tau\right)$.

\noindent 2) For the terms in the denominator of $\mathcal Q[\rho,M]$ we note that
$$
\int_{\R^N}\left( \rho_*(x)+ \epsilon^{-N} \tau\,\sigma(x/\epsilon) \right)dx+ (M_*-\tau) = \int_{\R^N} \rho_*\,dx+ M_*
$$
and, by Lemma~\ref{bl} applied with $p=1$,
$$
\int_{\R^N}\left( \rho_*(x)+ \epsilon^{-N} \tau\,\sigma(x/\epsilon) \right)^q\,dx = \int_{\R^N} \rho_*^q\,dx+ \epsilon^{N\,(1-q)} \tau^q \int_{\R^N} \sigma^q\,dx+ o\left(\epsilon^{N\,(1-q)}\tau^q\right)\,.
$$
Thus,
\begin{multline*}
\left( \int_{\R^N}\left( \rho_*(x)+ \epsilon^{-N} \tau\,\sigma(x/\epsilon) \right)^q\,dx \right)^{-\frac{2-\alpha}q}\\
=\left( \int_{\R^N} \rho_*^q\,dx \right)^{-\frac{2-\alpha}q}\left( 1- \frac{2-\alpha}q\,\epsilon^{N\,(1-q)} \tau^q\,\frac{\int_{\R^N} \sigma^q\,dx}{\int_{\R^N} \rho_*^q\,dx}+ R_2 \right)
\end{multline*}
with $R_2= o\left(\epsilon^{N\,(1-q)}\tau^q\right)$.

\smallskip Now we collect the estimates. Since $(\rho_*,M_*)$ is a minimizer, we obtain that
\begin{multline*}
\mathcal Q\left[\rho_*+\epsilon^{-N} \tau\,\sigma(\cdot/\epsilon),M_*-\tau\right] = \mathcal C_{N,\lambda,q}\left( 1- \frac{2-\alpha}{q}\,\epsilon^{N\,(1-q)} \tau^q\,\frac{\int_{\R^N} \sigma^q\,dx}{\int_{\R^N} \rho_*^q\,dx}+ R_2 \right) \\
+R_1\left( \int_{\R^N} \rho_*\,dx+ M_* \right)^{-\alpha}\left( \int_{\R^N}\left( \rho_*(x)+ \epsilon^{-N} \tau\,\sigma(x/\epsilon) \right)^q\,dx \right)^{-\frac{2-\alpha}q}\,.
\end{multline*}
If $\beta=\min\{2,\lambda\}>N\,(1-q)$, we can choose $\tau$ to be a fixed number in $(0,M_*)$, so that $R_1 =o\left(\epsilon^{N\,(1-q)}\right)$ and therefore
$$
\mathcal Q\left[\rho_*+\epsilon^{-N} \tau\,\sigma(\cdot/\epsilon),M_*-\tau\right]
\le\mathcal C_{N,\lambda,q}\left( 1- \frac{2-\alpha}q\,\epsilon^{N\,(1-q)} \tau^q\,\frac{\int_{\R^N} \sigma^q\,dx}{\int_{\R^N} \rho_*^q\,dx}+ o\left(\epsilon^{N\,(1-q)}\right) \right)\,.
$$
Since $\alpha<2$, this is strictly less than $\mathcal C_{N,\lambda,q}$ for $\epsilon>0$ small enough, contradicting the definition of $\mathcal C_{N,\lambda,q}$ as an infimum. Thus, $M_*=0$.

Note that if either $N=1$, $2$ or if $N\ge3$ and $\lambda\le2\,N/(N-2)$, then the assumption $q>N/(N+\lambda)$ implies that $\beta>N\,(1-q)$. If $N\ge 3$ and $\lambda>2\,N/(N-2)$, then $\beta=2\ge N\,(1-q)$ by assumption. Thus, it remains to deal with the case where $N\ge 3$, $\lambda>2\,N/(N-2)$ and $2=N\,(1-q)$. In this case we have $R_1 = O\left(\epsilon^2\,\tau\right)$ and therefore
$$
\mathcal Q\left[\rho_*+\epsilon^{-N} \tau\,\sigma(\cdot/\epsilon),M_*-\tau\right]
\le\mathcal C_{N,\lambda,q}\left( 1- \frac{2-\alpha}q\,\epsilon^2\,\tau^q\,\frac{\int_{\R^N} \sigma^q\,dx}{\int_{\R^N} \rho_*^q\,dx}+ O\left(\epsilon^2\,\tau\right) \right)\,.
$$
By choosing $\tau$ small (but independent of $\epsilon$) we obtain a contradiction as before. This completes the proof of the proposition.
\end{proof}
\begin{remark}\label{rem:taylor}
The extra assumption $q\ge 1-2/N$ for $N\ge 3$ and $\lambda>2\,N/(N-2)$ is dictated by the $\epsilon^2$ bound on $R_1$. We claim that for any $\lambda\ge2$, this bound is optimal. Namely, one has
\begin{multline*}
\iint_{\R^N\times\R^N} \rho_*(x)\left(|x-y|^\lambda - |x|^\lambda \right) \epsilon^{-N} \sigma(y/\epsilon)\,dx\,dy \\
= \epsilon^2 \ \frac{\lambda}{2}\left( 1+ \frac{\lambda-2}{N} \right) \int_{\R^N} |x|^{\lambda-2}\rho_*(x)\,dx \int_{\R^N} |y|^2 \sigma(y)\,dy+ o\left(\epsilon^2\right)
\end{multline*}
for $\lambda\ge2$. This follows from the fact that, for any given $x\neq0$,
$$
|x-y|^\lambda - |x|^\lambda = -\lambda\,|x|^{\lambda-2} x\cdot y+ \frac{\lambda}{2}\,|x|^{\lambda-2}\left( |y|^2+ (\lambda-2) \frac{(x\cdot y)^2}{|x|^2} \right)+ O\left(|y|^{\min\{3,\lambda\}}+ |y|^\lambda\right)\,.
$$
\end{remark}

\section{Further results of regularity}\label{Sec:AdditionalResults}

In this section we discuss the existence of a minimizer for $\mathcal C_{N,\lambda,q}$ in the regime that is not covered by Proposition~\ref{opt2}. In particular, we will establish a connection between the regularity of minimizers for the relaxed problem $\mathcal C_{N,\lambda,q}^{\rm rel}$ and the presence or absence of a Dirac delta. This will allow us to establish existence of minimizers for $\mathcal C_{N,\lambda,q}$ in certain parameter regions which are not covered by Proposition~\ref{opt2}.
\begin{proposition}\label{prop:reg}
Let $N\ge3$, $\lambda>2\,N/(N-2)$ and $N/(N+\lambda)<q<\min\big\{1-2/N\,,\,2\,N/(2\,N+\lambda)\big\}$. If $(\rho_*,M_*)$ is a minimizer for $\mathcal C_{N,\lambda,q}^{\rm{rel}}$ such that $(\rho_*,M_*) \in \mathrm L^{N\,(1-q)/2}(\R^N)\times[0,+\,\infty) $, then $M_*=0$.
\end{proposition}
The condition that the minimizer $(\rho_*,M_*)$ of $\mathcal C_{N,\lambda,q}^{\rm{rel}}$ belongs to $\mathrm L^{N\,(1-q)/2}(\R^N)\times[0,+\,\infty)$ has to be understood as a \emph{regularity} condition on $\rho_*$.
\begin{proof}
We argue by contradiction assuming that $M_*>0$ and consider a test function $\left(\rho_*+ \epsilon^{-N} \tau_\epsilon\,\sigma(\cdot/\epsilon),M_*-\tau_\epsilon\right)$ such that $\irN\sigma=1$. We choose $\tau_\epsilon = \tau_1\,\epsilon^{N-2/(1-q)}$ with a constant $\tau_1$ to be determined below. As in the proof of Proposition~\ref{opt2}, one has
\begin{multline*}
I_\lambda\left[\rho_*+ \epsilon^{-N} \tau_\epsilon\,\sigma(\cdot/\epsilon)\right]+ 2(M_*-\tau_\epsilon) \int_{\R^N} |x|^\lambda\left(\rho_*(x)+ \epsilon^{-N} \sigma(x/\epsilon)\right) dx \\
= I_\lambda[\rho_*]+ 2M_*\Jlambda{\rho_*}x+ R_1
\end{multline*}
with $R_1=O\left(\epsilon^2 \tau_\epsilon\right)$. Note here that we have $\lambda\ge2$. For the terms in the denominator we note that
$$
\int_{\R^N}\left( \rho_*(x)+ \epsilon^{-N} \tau_\epsilon\,\sigma(x/\epsilon) \right)dx+ (M_*-\tau_\epsilon) = \int_{\R^N} \rho_*\,dx+ M_*
$$
and, by Lemma~\ref{bl} applied with $p=N\,(1-q)/2$ and $\tau=\tau_\epsilon$, \emph{i.e.}, $\epsilon^{-N}\tau_\epsilon=\epsilon^{-N/p}\tau_1$, we have
$$
\int_{\R^N}\left( \rho_*(x)+ \epsilon^{-N} \tau_\epsilon\,\sigma(x/\epsilon) \right)^q\,dx = \int_{\R^N} \rho_*^q\,dx+ \epsilon^{N\,(1-q)} \tau_\epsilon^q \int_{\R^N} \sigma^q\,dx+ o\left(\epsilon^{N\,(1-q)}\tau_\epsilon^q\right)\,.
$$
Because of the choice of $\tau_\epsilon$ we have
$$
\epsilon^{N\,(1-q)} \tau_\epsilon^q = \epsilon^\gamma \tau_1^q
\quad\text{and}\quad
\epsilon^2 \tau_\epsilon = \epsilon^\gamma \tau_1
\quad\text{with}\quad\gamma := \frac{N-q\,(N+2)}{1-q}>0
$$
and thus
$$
\mathcal Q\left[\rho_*+\epsilon^{-N} \tau_\epsilon\,\sigma(\cdot/\epsilon),M_*-\tau_\epsilon\right]
\le\mathcal C_{N,\lambda,q}\left( 1- \frac{2-\alpha}q\,\epsilon^\gamma \tau_1^q\,\frac{\int_{\R^N} \sigma^q\,dx}{\int_{\R^N} \rho_*^q\,dx}+ O\left(\epsilon^\gamma \tau_1\right) \right)\,.
$$
By choosing $\tau_1$ small (but independent of $\epsilon$) we obtain a contradiction as before.
\end{proof}

Proposition~\ref{prop:reg} motivates the study of the regularity of the minimizer $(\rho_*,M_*)$ of $\mathcal C_{N,\lambda,q}^{\rm{rel}}$. We are not able to prove the regularity required in Proposition~\ref{prop:reg}, but we can state a dichotomy result which is interesting by itself, and allows to deduce the existence of minimizers for $\mathcal C_{N,\lambda,q}$ in parameter regions not covered in Proposition~\ref{opt2}.
\begin{proposition}\label{prop:cases}
Let $N\ge1$, $\lambda>0$ and $N/(N+\lambda)<q<2\,N/(2\,N+\lambda)$.
Let $(\rho_*,M_*)$ be a minimizer for $\mathcal C_{N,\lambda,q}^{\rm{rel}}$. Then the following holds:
\begin{enumerate}
\item If $\int_{\R^N}\rho_*\,dx > \frac{\alpha}{2}\,\frac{I_\lambda[\rho_*]}{\Jlambda{\rho_*}x}$, then $M_*=0$ and $\rho_*$ is bounded with
$$
\rho_*(0) =\left( \frac{(2-\alpha) I_\lambda[\rho_*] \int_{\R^N} \rho_*\,dx}{\left(\int_{\R^N} \rho_*^q\,dx\right)\left(2\Jlambda{\rho_*}x\int_{\R^N} \rho_*\,dx - \alpha I_\lambda[\rho_*]\right)} \right)^{1/(1-q)}\,.
$$
\item If $\int_{\R^N}\rho_*\,dx = \frac{\alpha}{2}\,\frac{I_\lambda[\rho_*]}{\Jlambda{\rho_*}x}$, then $M_*=0$ and $\rho_*$ is unbounded.
\item If $\int_{\R^N}\rho_*\,dx < \frac{\alpha}{2}\,\frac{I_\lambda[\rho_*]}{\Jlambda{\rho_*}x}$, then $\rho_*$ is unbounded and
$$
M_* = \frac{\alpha I_\lambda[\rho_*] - 2\Jlambda{\rho_*}x\ \int_{\R^N} \rho_*\,dx }{2\,(1-\alpha)\Jlambda{\rho_*}x}>0\,.
$$
\end{enumerate}
\end{proposition}
To prove Proposition~\ref{prop:cases}, let us begin with an elementary lemma.
\begin{lemma}\label{lem:minM}
For constants $A$, $B>0$ and $0<\alpha<1$, define
$$
f(M) := \frac{A+M}{(B+M)^\alpha}
\quad\text{for any}\quad M\ge0\,.
$$
Then $f$ attains its minimum on $[0,\infty)$ at $M=0$ if $\alpha A\le B$ and at $M= (\alpha A-B)/(1-\alpha)>0$ if $\alpha A> B$.
\end{lemma}
\begin{proof}
We consider the function on the larger interval $(-B,\infty)$. Let us compute
$$
f'(M) = \frac{(B+M) - \alpha(A+M)}{(B+M)^{\alpha+1}} = \frac{B-\alpha A+ (1-\alpha)M}{(B+M)^{\alpha+1}}\,.
$$
Note that the denominator of the right side vanishes exactly at $M=(\alpha A-B)/(1-\alpha)$, except possibly if this number coincides with $-B$.

We distinguish two cases. If $A\le B$, which is the same as $(\alpha A-B)/(1-\alpha)\le-B$, then~$f$ is increasing on $(-B,\infty)$ and then $f$ indeed attains its minimum on $[0,\infty)$ at $0$. Thus it remains to deal with the other case, $A>B$. Then $f$ is decreasing on $\big(-B,(\alpha A-B)/(1-\alpha)\big]$ and increasing on $\big[(\alpha A-B)/(1-\alpha),\infty\big)$. Therefore, if $\alpha A-B\le0$, then $f$ is increasing on $[0,\infty)$ and again the minimum is attained at $0$. On the other hand, if $\alpha A-B>0$, then $f$ has a minimum at the positive number $M=(\alpha A-B)/(1-\alpha)$.
\end{proof}

\begin{proof}[Proof of Proposition~\ref{prop:cases}.]
\emph{Step 1.} We vary $\mathcal Q[\rho_*,M]$ with respect to $M$. By the minimizing property of $M_*$ the function
$$
M\mapsto \mathcal Q[\rho_*,M] = \frac{2\Jlambda{\rho_*}x}{\left( \int_{\R^N} \rho_*^q\,dx \right)^{(2-\alpha)/q}}\,\frac{A+M}{(B+M)^\alpha}
$$
with
$$
A:= \frac{I_\lambda[\rho_*]}{2\Jlambda{\rho_*}x}
\quad\text{and}\quad
B := \int_{\R^N}\rho_*(x)\,dx
$$
attains its minimum on $[0,\infty)$ at $M_*$. From Lemma~\ref{lem:minM} we infer that
$$
M_* = 0
\quad\text{if and only if}\quad
\frac{\alpha}{2}\,\frac{I_\lambda[\rho_*]}{\Jlambda{\rho_*}x}\le\int_{\R^N}\rho_*(x)\,dx\,,
$$
and that $M_* = \frac{\alpha I_\lambda[\rho_*] - 2\left(\Jlambda{\rho_*}x\right)\left(\int_{\R^N} \rho_*(y)\,dy\right)}{2\,(1-\alpha)\Jlambda{\rho_*}x}$ if $\frac{\alpha}{2}\,\frac{I_\lambda[\rho_*]}{\Jlambda{\rho_*}x} > \int_{\R^N}\rho_*(x)\,dx$.

\smallskip\noindent\emph{Step 2.} We vary $\mathcal Q[\rho,M_*]$ with respect to $\rho$. Letting $x\to 0$ in the Euler--Lagrange equation~\eqref{EL}, we find that
$$
2\,\frac{\Jlambda{\rho_*}y}{I_\lambda[\rho_*]+ 2M_*\Jlambda{\rho_*}y} - \alpha\,\frac{1}{\int_{\R^N} \rho_*(y)\,dy+M_*} = (2-\alpha)\,\frac{\rho_*(0)^{-1+q}}{\int_{\R^N} \rho_*(y)^q\,dy}\ge0\,,
$$
with the convention that the last inequality is an equality if and only if $\rho_*$ is unbounded. Consistently, we shall write that $\rho_*(0)=+\infty$ in that case. We can rewrite our inequality as
$$
M_*\ge\frac{\alpha I_\lambda[\rho_*] - 2\left(\Jlambda{\rho_*}y\right)\left( \int_{\R^N} \rho_*\,dy\right)}{2\,(1-\alpha)\Jlambda{\rho_*}y}
$$
with equality if and only if $\rho_*$ is unbounded. This completes the proof of Proposition~\ref{prop:cases}.
\end{proof}

Next, we focus on matching ranges of the parameters $(N,\lambda, q)$ with the cases (1), (2) and (3) in Proposition~\ref{prop:cases}. For any $\lambda\ge 1$ we deduce from
\begin{equation}
\label{eq:kernelupperbound}
|x-y|^\lambda\le\big(|x|+|y|\big)^\lambda\le2^{\lambda-1}\,\big(|x|^\lambda+|y|^\lambda\big)
\end{equation}
that
$$
I_\lambda[\rho]<2^\lambda\Jlambda\rho x\int_{\R^N} \rho(x)\,dx\,.
$$
For all $\alpha\le2^{-\lambda+1}$, which can be translated into
$$
q\ge\frac{2\,N\,\big(1-2^{-\lambda}\big)}{2\,N\,\big(1-2^{-\lambda}\big)+\lambda}\,,
$$
that is,
$$
\int_{\R^N}\rho_*\,dx\ge\frac{\alpha}{2}\,\frac{I_\lambda[\rho_*]}{\Jlambda{\rho_*}x}\,,
$$
so that Cases~(1) and (2) of Proposition~\ref{prop:cases} apply and we infer that $M_*=0$. Note that this bound for $q$ is in the range $\big(N/(N+\lambda)\,,\,2\,N/(2\,N+\lambda)\big)$ for all \hbox{$\lambda\ge 1$}. See Fig.~\ref{Fig1}.

A better range for which $M_*=0$ can be obtained for $N\ge3$ using the fact that superlevel sets of a symmetric non-increasing function are balls. From the layer cake representation we deduce that
$$
\label{lcr}
I_\lambda[\rho] \leq 2\,A_{N,\lambda}\Jlambda\rho x\int_{\R^N} \rho(x)\,dx\,, \quad
A_{N,\lambda}:=\sup_{0\leq R,S<\infty}F(R,S)\,,
$$
where
$$
F(R,S):=\frac{\iint_{B_R\times B_S} |x-y|^\lambda\,dx\,dy}{|B_R| \int_{B_S} |x|^\lambda\,dx+ |B_S| \int_{B_R} |y|^\lambda\,dy}\,.
$$
For any $\lambda\ge1$, we have $2\,A_{N,\lambda}\le2^\lambda$ by~\eqref{eq:kernelupperbound}, and also $A_{N,\lambda}\ge1/2$ because by~\eqref{eq:lowerboundsymmdecr} $I_\lambda[\1_{B_1}]\ge|B_1|\int_{B_1}|y|^\lambda\,dy$. Note that $A_{N,2}=1$ since, for $\lambda=2$ and for any $R$, $S>0$, $F(R,S)=1$ by expanding the square in the numerator. The bound $A_{N,\lambda}\ge1/2$ can be improved to $A_{N,\lambda}>1$ for any $\lambda>2$ as follows. We know that
$$
A_{N,\lambda}\ge F(1,1)=\frac{N\,(N+\lambda)}2\kern-2pt\iint_{0\le r,\,s\le1}\kern-24pt r^{N-1}\,s^{N-1}\left(\int_0^\pi\left(r^2+s^2-2\,r\,s\,\cos\phi\right)^{\lambda/2}\frac{(\sin\phi)^{N-2}}{W_N}\,d\phi\right)dr\,ds
$$
with the Wallis integral $W_N:=\int_0^\pi(\sin\phi)^{N-2}\,d\phi$. For any $\lambda>2$, we can apply Jensen's inequality twice and obtain
\begin{multline*}
\int_0^\pi\left(r^2+s^2-2\,r\,s\,\cos\phi\right)^{\lambda/2}\frac{(\sin\phi)^{N-2}\,d\phi}{W_N}\\
\ge\left(\int_0^\pi\left(r^2+s^2-2\,r\,s\,\cos\phi\right)\frac{(\sin\phi)^{N-2}\,d\phi}{W_N}\right)^{\lambda/2}=\left(r^2+s^2\right)^{\lambda/2}
\end{multline*}
and
\begin{multline*}
\iint_{0\le r,\,s\le1}\kern-18pt r^{N-1}\,s^{N-1}\left(r^2+s^2\right)^{\lambda/2}dr\,ds\\
\ge\frac1{N^2}\left(\iint_{0\le r,\,s\le1}\kern-18pt r^{N-1}\,s^{N-1}\left(r^2+s^2\right)N^2\,dr\,ds\right)^{\lambda/2}=\frac1{N^2}\,\left(\frac{2\,N}{N+2}\right)^{\lambda/2}\,.
\end{multline*}
Hence
$$
A_{N,\lambda}\ge\frac{N+\lambda}{2\,N}\,\left(\frac{2\,N}{N+2}\right)^{\lambda/2}=:B_{N,\lambda}
$$
where $\lambda\mapsto B_{N,\lambda}$ is monotone increasing, so that $A_{N,\lambda}\ge B_{N,\lambda}>B_{N,2}=1$ for any $\lambda>2$. In this range we can therefore define
\be{qbar}
\bar q(\lambda,N):=\frac{2\,N\left(1-\frac1{2A_{N,\lambda}}\right)}{2\,N\left(1-\frac1{2A_{N,\lambda}}\right)+\lambda}\,.
\ee
Based on a numerical computation, the curve $\lambda\mapsto\bar q(\lambda,N)$ is shown on Fig.~\ref{Fig1}. Note that in the case $\lambda=2$, the curve $\bar q(\lambda,N)$ coincides with $N/(N+\lambda)$. The next result summarizes our considerations above.
\begin{proposition}\label{Cor:Mstar} Assume that $N\ge3$. Then $\bar q$ defined by~\eqref{qbar} is such that
$$
\bar q(\lambda,N)\le\frac{2\,N\,\big(1-2^{-\lambda}\big)}{2\,N\,\big(1-2^{-\lambda}\big)+\lambda}<\frac{2\,N}{2\,N+\lambda}\quad\mbox{for}\quad\lambda\ge 1\quad\mbox{and}\quad\bar q(\lambda,N)>\frac N{N+\lambda}\quad\mbox{for}\quad\lambda>2\,.
$$
If $(\rho_*,M_*)$ is a minimizer for $\mathcal C_{N,\lambda,q}^{\rm{rel}}$ and if $\max\left\{\bar q(\lambda,N),\frac N{N+\lambda}\right\}<q<\frac{N-2}N$, then $M_*=0$ and $\rho_*$ is bounded.\end{proposition}
\noindent Notice that $\frac N{N+\lambda}<\frac{N-2}N$ means $\lambda>\frac{2\,N}{N-2}$. We recall that the case $q\ge\frac{N-2}N$ has been covered in Proposition~\ref{opt2}.
\begin{proof} We recall that $q>\bar q(\lambda,N)$ defined by~\eqref{qbar} means that
$$
\int_{\R^N}\rho_*\,dx>\frac{\alpha}{2}\,\frac{I_\lambda[\rho_*]}{\Jlambda{\rho_*}x}\,,
$$
so that Case~(1) of Proposition~\ref{prop:cases} applies. The estimates on $\bar q$ follow from elementary computations.\end{proof}
Next we consider the singularity of $\rho_*$ at the origin in the unbounded case in more detail, in the cases which are not already covered by Propositions~\ref{opt0},~\ref{opt2} and~\ref{Cor:Mstar}.
\begin{lemma}\label{nearorigin}
Let $N\ge3$, $\lambda>2\,N/(N-2)$ and $N/(N+\lambda)<q<\min\left\{1-N/2,\bar q(\lambda,N)\right\}$. Let $(\rho_*,M_*)$ be a minimizer for $\mathcal C_{N,\lambda,q}^{\rm rel}$ and assume that it is unbounded. Then there is a constant $C>0$ such that
$$
\rho_*(x) = C\,|x|^{-2/(1-q)}\,\big(1+o(1)\big)\quad\text{as}\quad x\to 0\,.
$$
\end{lemma}
\begin{proof}
Since $\rho_*(x)\to\infty$ as $x\to 0$ we can rewrite the Euler--Lagrange equation~\eqref{EL} as
$$\label{ELunbb}
2\,\frac{\int_{\R^N}\left(|x-y|^\lambda-|y|^\lambda\right) \rho_*(y)\,dy+ M_* |x|^\lambda}{I_\lambda[\rho_*]+ 2M_*\Jlambda{\rho_*}y} - (2-\alpha)\,\frac{\rho_*(x)^{-1+q}}{\int_{\R^N} \rho_*(y)^q\,dy} = 0\,.
$$
By Taylor expanding we have
$$
\int_{\R^N} \left(|x-y|^\lambda-|y|^\lambda\right) \rho_*(y)\,dy+ M_*\,|x|^\lambda = C_1\,|x|^2\,\big(1+o(1)\big)
\quad\text{as}\quad x\to 0
$$
with $C_1=\frac12\,\lambda\,(\lambda-1) \int_{\R^N} |y|^{\lambda-2} \rho_*(y)\,dy$, which is finite according to~\eqref{eq:lowerboundsymmdecr}. This gives the claimed behavior for $\rho_*$ at the origin.
\end{proof}
The proof of Lemma~\ref{nearorigin} relies only on~\eqref{EL}. For this reason, we can also state the following result.
\begin{proposition}\label{bounded}
Let $N\ge 1$, $\lambda>0$ and $N/(N+\lambda)<q<1$. If $N\ge 3$ and $\lambda>2\,N/(N-2)$ we assume in addition that $q\ge\min\left\{1-N/2,\bar q(\lambda,N)\right\}$. If $(\rho_*,M_*)\in\mathrm L^1\cap\mathrm L^q\left(\R^N\right)\cap\mathrm L^1\left(\R^N,|x|^\lambda\,dx\right)\times\R^+$ solves~\eqref{EL}, then $M_*=0$ and $\rho_*$ is bounded.
\end{proposition}
 As a consequence, under the assumptions of Proposition~\ref{bounded}, we recover that any minimizer $(\rho_*,M_*)$ of $\mathcal C_{N,\lambda,q}^{\rm rel}$ is such that $M_*=0$ and $\rho_*$ is bounded. Notice that the range $\bar q(\lambda,N)<1-2/N$ is covered in Proposition~\ref{Cor:Mstar} but not here.
\begin{proof}
Assume by contradiction that $\rho_*$ is unbounded. If $\lambda\ge2$, the proof of Lemma~\ref{nearorigin} applies and we know that $\rho_*(x)\sim |x|^{-2/(1-q)}$ as $x\to0$. For any $\lambda\in(0,1]$ we have that $|x-y|^\lambda\leq |x|^\lambda+|y|^\lambda$. If $\lambda\in(1,2)$, using inequality~\eqref{eq:elemineq} with the roles of $x$ and $y$ interchanged, we find that $\int_{\R^N} \left(|x-y|^\lambda-|y|^\lambda\right) \rho_*(y)\,dy \leq C\,|x|^\lambda$ for some $C>0$. Hence, for some $c>0$,
$$
\label{eq:inequnbdd0}
\rho_*(x) \geq c\,|x|^{-\min\{\lambda,2\}/(1-q)}
$$
for any $x\in\R^N$ with $|x|>0$ small enough. We claim that $\min\{\lambda,2\}/(1-q)\ge N$, which contradicts $\int_{\R^N}\rho_*\,dx<\infty$.
\end{proof}

By recalling the results of~\cite{DZ15} in the conformally invariant case $q=2N/(2N+\lambda)$, and the results of Propositions~\ref{ineq},~\ref{opt0},~\ref{opt2},~\ref{bounded} and Lemma~\ref{sym-pos}, we have completed the proof of Theorem~\ref{main}.

\section{Free Energy}\label{sec:freeenergy}

In this section, we discuss the relation between the reverse HLS inequalities~\eqref{ineq:rHLS} and the free energy functional
$$
\mathcal F[\rho]:=-\frac{1}{1-q}\int_{\R^N}\rho^q\,dx+\frac{1}{2\lambda}\,I_\lambda[\rho]\,.
$$
We also extend the free energy functional to the set of probability measures and prove a uniqueness result in this framework.

\subsection{Relaxation and extension of the free energy functional}\label{Free Energy: definitions}

The kernel $|x-y|^\lambda$ is positive and continuous, so there is no ambiguity with the extension of $I_\lambda$ to $\mathcal{P}(\R^N)$, which is simply given by
$$
I_\lambda[\mu] = \iint_{\R^N\times\R^N} |x-y|^\lambda\,d\mu(x)\,d\mu(y)\,.
$$
In this section we use the notion of weak convergence in the sense of probability theory: if $\mu_n$ and $\mu$ are probability measures on $\R^N$ then $\mu_n\rightharpoonup\mu$ means $\int_{\R^N} \phi\,d\mu_n \to \int_{\R^N} \phi\,d\mu$ for all bounded continuous functions $\phi$ on $\R^N$. We define the \emph{extension} of $\mathcal{F}$ to $\mathcal{P}(\R^N)$ by
$$\label{eq:relaxation}
\mathcal{F}^\Gamma[\mu]:=\inf_{\substack{(\rho_n)_{n\in\N}\subset C_c^\infty\cap\mathcal{P}(\R^N)\\ \mathrm{s.t.}\;\rho_n\rightharpoonup\mu}}\liminf_{n\to\infty}\mathcal{F}[\rho_n]\,.
$$
We also define a \emph{relaxed free energy} by
$$
\mathcal F^{\rm rel}[\rho,M]:=-\,\frac1{1-q}\int_{\R^N}\rho(x)^q\,dx+\frac1{2\,\lambda}\,I_\lambda[\rho]+\frac M\lambda\int_{\R^N}|x|^\lambda\,\rho(x)\,dx\,.
$$
The functional $\mathcal F^{\rm rel}$ can be characterized as the restriction of $\mathcal F^\Gamma$ to the subset of probability measures whose singular part is a multiple of a $\delta$ at the origin.

\subsection{Equivalence of the optimization problems and consequences}According to Pro\-position~\ref{ineq}, we know that $\mathcal C_{N,\lambda,q}=0$ if $0<q\le N/(N+\lambda)$, so that one can find a sequence of test functions $\rho_n\in\mathrm L^1_+\cap\mathrm L^q(\R^N)$ such that
$$
\|\rho_n\|_{\mathrm L^1(\R^N)}=I_\lambda[\rho_n]=1\quad\mbox{and}\quad\int_{\R^N}\rho_n(x)^q\,dx\geq n\in\N\,.
$$
As a consequence, $\lim_{n\to\infty}\mathcal{F}[\rho_n]=-\,\infty$. 

Next, let us consider the case $N/(N+\lambda)<q<1$. Assume that $\rho\in\mathrm L^1_+\cap\mathrm L^q(\R^N)$ is such that $I_\lambda[\rho]$ is finite. For any $\ell>0$ we define $\rho_\ell(x):=\ell^{-N}\,\rho(x/\ell)/\|\rho\|_{\mathrm L^1(\R^N)}$ and compute
$$
\mathcal F[\rho_\ell]=-\,\ell^{(1-q)\,N}\,\mathsf A+\ell^\lambda\,\mathsf B
$$
where $\mathsf A=\frac1{1-q}\int_{\R^N}\rho(x)^q\,dx/\|\rho\|_{\mathrm L^1(\R^N)}^q$ and $\mathsf B=\frac1{2\lambda}\,I_\lambda[\rho]/\|\rho\|_{\mathrm L^1(\R^N)}^2$. The function $\ell\mapsto\mathcal F[\rho_\ell]$ has a minimum which is achieved at $\ell=\ell_\star$ where
$$
\ell_\star:=\left(\tfrac{N\,(1-q)\,\mathsf A}{\lambda\,\mathsf B}\right)^\frac1{\lambda-N\,(1-q)}
$$
and, with $\mathsf Q_{q,\lambda}$ as defined in~\eqref{newQ}, we obtain that
$$
\mathcal F[\rho]\ge\mathcal F[\rho_{\ell_\star}]=-\,\kappa_\star\left(\mathsf Q_{q,\lambda}[\rho]\right)^{-\frac{N\,(1-q)}{\lambda-N\,(1-q)}}\quad\mbox{where}\quad\kappa_\star:=\tfrac{\lambda-N\,(1-q)}{(1-q)\,\lambda}\,(2\,N)^\frac{N\,(1-q)}{\lambda-N\,(1-q)}\,.
$$
As a consequence, we have the following result.
\begin{proposition}\label{equiv} With the notations of Section~\ref{Free Energy: definitions}, for any $q\in(0,1)$ and $\lambda>0$, we have
$$
F_{N,\lambda,q}:=\inf_\rho\mathcal F[\rho]=\inf_{\rho,M}\mathcal F^{\rm rel}[\rho,M]=\inf_\mu\mathcal{F}^\Gamma[\mu]
$$
where the infima are taken on $\mathrm L^1_+\cap\mathrm L^q(\R^N)$, $\left(\mathrm L^1_+\cap\mathrm L^q(\R^N)\right)\times[0,\infty)$ and $\mathcal{P}(\R^N)$ in case of, respectively, $\mathcal F$, $\mathcal F^{\rm rel}$ and $\mathcal{F}^\Gamma$. Moreover $F_{N,\lambda,q}>-\infty$ if and only if $\mathcal C_{N,\lambda,q}>0$, that is, if $N/(N+\lambda)<q<1$ and, in this case,
$$
F_{N,\lambda,q}=-\,\kappa_\star\,\mathcal C_{N,\lambda,q}^{-\frac{N\,(1-q)}{\lambda-N\,(1-q)}}=\mathcal F^{\rm rel}[\rho_*,M_*]=\mathcal{F}^\Gamma[\mu_*]
$$
for some $\mu_*= M_*\,\delta+\rho_*$, $(\rho_*,M_*)\in\left(\mathrm L^1_+\cap\mathrm L^q(\R^N)\right)\times[0,1)$ such that $\int_{\R^N}\rho_*(x)\,dx+M_*=1$. Additionally, we have that
$$
I_\lambda[\rho_*]+2\,M_*\int_{\R^N}|x|^\lambda\,\rho_*(x)\,dx=2\,N\int_{\R^N}\rho_*(x)^q\,dx\,.
$$
Since $(\rho_*,M_*)$ is also a minimizer for $\mathcal C_{N,\lambda,q}^{\rm{rel}}$, it satisfies all properties of Lemma~\ref{sym-pos} and Propositions~\ref{opt2},~\ref{prop:cases} and~\ref{Cor:Mstar}.
\end{proposition}
\begin{proof} This result is a simple consequence of the definitions of $\mathcal F^{\rm rel}$ and $\mathcal{F}^\Gamma$. The existence of the minimizer is a consequence of Propositions~\ref{opt0} and~\ref{opt1}. If $\rho\in\mathrm L^1_+\cap\mathrm L^q(\R^N)$ is a minimizer for $F_{N,\lambda,q}$, then $I_\lambda[\rho]=2\,N\int_{\R^N}\rho(x)^q\,dx$ because \hbox{$\ell_\star=1$}, and $\rho$ is also an optimizer for $\mathcal C_{N,\lambda,q}$. Conversely, if $\rho\in\mathrm L^1_+\cap\mathrm L^q(\R^N)$ is an optimizer for $\mathcal C_{N,\lambda,q}$, then there is an $\ell>0$ such that $\ell^{-N}\,\rho(\cdot/\ell)/\|\rho\|_{\mathrm L^1(\R^N)}$ is an optimizer for $F_{N,\lambda,q}$.\end{proof}

The discussion of whether $M_*=0$ or not in the statement of Proposition~\ref{equiv} is the same as in the discussion of the reverse Hardy--Littlewood--Sobolev inequality in Section~\ref{Sec:Existence}. Except for the question of uniqueness, this completes the proof of Theorem~\ref{FreeEnergy}.

\subsection{Properties of the free energy extended to probability measures}
From now on, unless it is explicitly specified, we shall denote by $\rho$ the absolutely continuous part of the measure $\mu\in\mathcal{P}(\R^N)$. On $\mathcal{P}(\R^N)$, let us define
\be{G}
\mathcal G[\mu]:=\frac{1}{2\lambda}\,I_\lambda[\mu]-\frac{1}{1-q}\int_{\R^N}\rho(x)^q\,dx
\ee
if $I_\lambda[\mu]<+\infty$ and extend it with the convention that $\mathcal G[\mu]=+\infty$ if $I_\lambda[\mu]=+\infty$. Notice that $\int_{\R^N}\rho(x)^q\,dx$ is finite by Lemma~\ref{ineq2} and Eq.~\eqref{eq:lowerboundsymmdecr} whenever $I_\lambda[\rho]\le I_\lambda[\mu]$ is finite. Let us start with some technical estimates. The following is a variation of~\cite[Lemma~2.7]{CDP}.
\begin{lemma}\label{momentbound}
Let $N\geq 1$ and $\lambda>0$, then for any $a\in\R^N$, $r>0$ and $\mu\in\mathcal P(\R^N)$ we have
$$
I_\lambda[\mu] \geq 2^{1-(\lambda-1)_+}\,\mu\big(B_r(a)\big)\left( \int_{\R^N} |y-a|^\lambda\,d\mu(y) - 2^{(\lambda-1)_+}\,r^\lambda \right)\,.
$$
As a consequence, if $I_\lambda[\mu]<\infty$, then $\int_{\R^N}|y-a|^\lambda\,d\mu(y)$ is finite for any $a\in\R^N$ and the infimum with respect to $a$ is achieved.
\end{lemma}
\begin{proof} If $x\in B_r(a)$ and $y\in B_r(a)^c$, then
$$
|x-y|^\lambda \geq \big(|y-a|-|x-a|\big)^\lambda \geq \big(|y-a| - r\big)^\lambda \geq 2^{-(\lambda-1)_+} |y-a|^\lambda - r^\lambda\,.
$$
We can therefore bound $I_\lambda[\mu]$ from below by
\begin{align*}
2&\iint_{B_r(a)\times B_r(a)^c}|x-y|^\lambda\,d\mu(x)\,d\mu(y) \\
& \geq 2\,\mu\big(B_r(a)\big)\left( 2^{-(\lambda-1)_+} \int_{B_r(a)^c} |y-a|^\lambda\,d\mu(y) - r^\lambda\,\mu\big(B_r(a)^c\big) \right) \\
&\kern12pt = 2^{1-(\lambda-1)_+}\,\mu\big(B_r(a)\big)\left( \int_{\R^N} |y-a|^\lambda\,d\mu(y) - \int_{B_r(a)} |y-a|^\lambda\,d\mu(y) - 2^{(\lambda-1)_+}\,r^\lambda\,\mu\big(B_r(a)^c\big) \right) \\
&\kern12pt \geq 2^{1-(\lambda-1)_+}\,\mu\big(B_r(a)\big)\left( \int_{\R^N} |y-a|^\lambda\,d\mu(y) - r^\lambda\,\mu\big(B_r(a)\big) - 2^{(\lambda-1)_+}\,r^\lambda\,\mu\big(B_r(a)^c\big) \right) \\
&\kern12pt \geq 2^{1-(\lambda-1)_+}\,\mu_n\big(B_r(a)\big)\left( \int_{\R^N} |y-a|^\lambda\,d\mu_n(y) - 2^{(\lambda-1)_+}\,r^\lambda \right)\,.
\end{align*}
This proves the claimed inequality. 

Let $R>0$ be such that $\mu\big(B_R(0)\big)\ge1/2$ and consider $a\in B_R(0)^c$, so that $|y-a|>|a|-R$ for any $y\in B_R(0)$. From the estimate
$$
\int_{\R^N}|y-a|^\lambda\,d\mu(y)\ge\int_{B_R(0)}|y-a|^\lambda\,d\mu(y)\ge\frac12\,\big(|a|-R\big)^\lambda\,,
$$
we deduce that in $\inf_{a\in\R^N}\int_{\R^N}|y-a|^\lambda\,d\mu(y)$, $a$ can be restricted to a compact region of $\R^N$. Since the map $a\mapsto\int_{\R^N}|y-a|^\lambda\,d\mu(y)$ is lower semi-continuous, the infimum is achieved.\end{proof}
\begin{corollary}\label{apriori}
Let $\lambda>0$ and $N/(N+\lambda)<q<1$. Then there is a constant $C>0$ such that
$$
\mathcal G[\mu]\ge\frac{I_\lambda[\mu]}{4\lambda} - C\geq \frac1{4\lambda}\inf_{a\in\R^N} \int_{\R^N} |x-a|^\lambda\,d\mu(x) - C\quad\forall\,\mu\in\mathcal{P}(\R^N)\,.
$$
\end{corollary}
\begin{proof} Let $\mu\in\mathcal{P}(\R^N)$ and let $\rho$ be its absolutely continuous part with respect to Lebes\-gue's measure. By Theorem~\ref{main}, we know that
$$
\int_{\R^N}\rho(x)^q\,dx\le\left(\frac{I_\lambda[\rho]}{\mathcal C_{N,\lambda,q}}\right)^\frac{N\,(1-q)}\lambda
$$
because $\int_{\R^N} \rho\,dx\le\mu(\R^N)=1$. Hence we obtain that
$$
\mathcal G[\mu]\ge\frac{I_\lambda[\mu]}{4\lambda}-C\quad\mbox{with}\quad C=\min\left\{\frac X{4\lambda}-\left(\frac X{\mathcal C_{N,\lambda,q}}\right)^\frac{N\,(1-q)}\lambda\,:\,X>0\right\}\,.
$$
As $\mu$ is a probability measure, the proof is completed using the inequality
\begin{equation*}
\inf_{a\in\R^N} \int_{\R^N} |x-a|^\lambda\,d\mu(x)\le \iint_{\R^N\times\R^N} |x-a|^\lambda\,d\mu(x)\,d\mu(a)=I_\lambda[\mu]\,.
\end{equation*}
\end{proof}
\begin{lemma}\label{lsc} If $\lambda>0$ and $N/(N+\lambda)<q<1$, then $\mathcal G$ is lower semi-continuous. \end{lemma}
\begin{proof} Let $(\mu_n)\subset\mathcal P(\R^N)$ with $\mu_n\rightharpoonup\mu$. We denote by $\rho_n$ and $\rho$ the absolutely continuous part of $\mu_n$ and $\mu$, respectively. We have to prove that $\liminf_{n\to\infty} \mathcal G[\mu_n] \geq \mathcal G[\mu]$. Either $\liminf_{n\to\infty} \mathcal G[\mu_n]=+\infty$, or it is finite and then, up to the extraction of a subsequence, we know from Corollary~\ref{apriori} that $\mathcal K:=\sup_{n\in\N}I_\lambda[\mu_n]$ is finite. According to~\cite[Proposition 7.2]{santambrogio2015optimal}, we also know that
$$\liminf_{n\to\infty} I_\lambda[\mu_n] \geq I_\lambda[\mu]\,.
$$
According to~\cite[Theorem 7.7]{santambrogio2015optimal} or \cite[Theorem 4]{BV88}, for any $r>0$ we have
$$
\liminf_{n\to\infty}\left(-\int_{\overline{B_r}}\,\rho_n(x)^q\,dx\right)\geq-\int_{\overline{B_r}}\,\rho(x)^q\,dx\,.
$$
Notice that the absolutely continuous part of the limit of $\mu_n\measurerestr\overline{B_r}$ coincides with the absolutely continuous part of $\mu\,\measurerestr\overline{B_r}$ as the difference is supported on $\partial B_r$.

We choose $r_0>0$ to be a number such that $\mu(B_{r_0})\ge1/2$ and find $n_0\in\N$ such that for any $n\geq n_0$ we have $\mu_n(B_{r_0})\ge1/4$. By applying Lemma~\ref{momentbound}, we obtain that
$$
\int_{\R^N}|x|^\lambda\,d\mu_n(x)\le2^{(\lambda-1)_+}\left(r_0^\lambda+2\,I_\lambda[\mu_n]\right)\le2^{(\lambda-1)_+}\big(r_0^\lambda+2\,\mathcal K\big)
$$
for any $n\ge n_0$. We apply Lemma~\ref{ineq2} to $\rho=\rho_n\,\1_{B_r^c}$
$$
\int_{B_r^c}\rho_n(x)^q\,dx\le c_{N,\lambda,q}^{-q}\left(\int_{B_r^c}\rho_n\,dx \right)^{q-\frac{N(1-q)}{\lambda}}\left( \int_{B_r^c} |x|^\lambda\,\rho_n\,dx \right)^{\frac{N(1-q)}{\lambda}}
$$
and conclude that
$$
\liminf_{n\to\infty}\left( - \int_{B_r^c}\rho_n(x)^q\,dx \right) \geq -\,c_{N,\lambda,q}^{-q}\left( \mu\big(B_r^c\big)\right)^{q-\frac{N(1-q)}{\lambda}}\left(2^{(\lambda-1)_+}\big(r_0^\lambda+2\,\mathcal K\big)\right)^{\frac{N(1-q)}{\lambda}}\,.
$$
The right hand side vanishes as $r\to\infty$, which proves the claimed lower semi-continuity.
\end{proof}
After these preliminaries, we can now prove that $\mathcal G$, defined in~\eqref{G}, is the lower-semi\-continuous envelope of $\mathcal F$. The precise statement goes as follows.
\begin{proposition}\label{def:relaxation}
Let $0<q<1$ and $\lambda>0$. Let $\mu\in\mathcal P(\R^N)$\begin{enumerate}
\item If $q\le N/(N+\lambda)$, then $\mathcal{F}^\Gamma[\mu]=-\,\infty$.
\item If $q> N/(N+\lambda)$, then $\mathcal{F}^\Gamma[\mu]=\mathcal G[\mu]$.
\end{enumerate}
\end{proposition}
\begin{proof} Assume that $q\le N/(N+\lambda)$. Using the function $\nu(x)=|x|^{-N-\lambda}\,\big(\log|x|)\big)^{-1/q}$, let us construct an approximation of any measure in $\mu\in\mathcal{P}(\R^N)$ given by a sequence $(\rho_n)_{n\in\N}$ of functions in $C^\infty_c\cap\mathcal{P}(\R^N)$ such that $\lim_{n\to\infty}\mathcal{F}[\rho_n]=-\,\infty$.

Let $\eta\in C^\infty_c(B_1)$ be a positive mollifier with unit mass and $\zeta\in C^\infty_c(B_2)$ be a cutoff function such that $\1_{B_1}\le\zeta\le1$. Given any natural numbers $i$, $j$ and $k$, we define $\eta_{i}(y):=i^N\,\eta(iy)$, $\zeta_{j}(y):=\zeta(y/j)$ and
\begin{equation*}
f_{i,j,k}:=\left(1-\tfrac1k\right)\left(\mu*\eta_{i} \right) \zeta_{i}+\tfrac1k\,C_{i,j,k}\,(1-\zeta_{i})\,\zeta_{j}\,\nu
\end{equation*}
where $C_{i,j,k}$ is a positive constant that has been picked so that $f_{i,j,k}\in\mathcal{P}(\R^N)$. We choose $i=n$, $j=e^n$ and $k=k(n)$ such that
$$
\lim_{n\to\infty}k(n)=+\infty\quad\mbox{and}\quad\lim_{n\to\infty}k(n)^{-N\,q}\,\log(n/\log n)=+\infty\,.
$$
By construction, $\rho_n:=f_{n,j(n),k(n)}\rightharpoonup \mu$ as $n\to\infty$ and $\lim_{n\to\infty}\mathcal{F}[\rho_n]=-\,\infty$, so $\mathcal{F}^\Gamma[\mu]=-\,\infty$.

\medskip Assume that $q>N/(N+\lambda)$ and consider a sequence of functions in $C^\infty_c\cap\mathcal{P}(\R^N)$ such that $\rho_n\rightharpoonup\mu$ and $\lim_{n\to\infty}\mathcal F[\rho_n]=\mathcal{F}^\Gamma[\mu]$. If $I_\lambda[\mu]=\infty$, by the lower-semicontinuity of~$I_\lambda$ (see for instance~\cite[Proposition 7.2]{santambrogio2015optimal}), we know that $\lim_{n\to\infty}I_\lambda[\rho_n]=\infty$ and deduce from Corollary~\ref{apriori} that $\frac1{4\lambda}\,I_\lambda[\rho_n] - C\le\mathcal F[\rho_n]$ diverges, so that $\mathcal{F}^\Gamma[\mu]=\infty=\mathcal G[\mu]$.

Next, we assume that $I_\lambda[\mu]<\infty$. According to Lemma~\ref{lsc}, we deduce from the lower semi-continuity of $\mathcal G$ that
\[
\mathcal F^\Gamma[\mu]=\lim_{n\to\infty}\mathcal F[\rho_n]=\lim_{n\to\infty}\mathcal G[\rho_n]\ge\mathcal G[\mu]\,.
\]
It remains to show the inequality $\mathcal F^\Gamma[\mu] \leq \mathcal G[\mu]$. Let $\mu_R:=\mu(B_R)^{-1}\, \mu\,\measurerestr B_{R}$. We have that $\mu_R\rightharpoonup\mu$ as $R\to\infty$ and, by monotone convergence,
$$
\lim_{R\to\infty} \mathcal G[\mu_R] = \mathcal G[\mu]\,.
$$
Let $\eta_\epsilon(x) := \epsilon^{-N}\, \eta(x/\epsilon)$ for a sufficiently regular, compactly supported, nonnegative function $\eta$ such that $\int_{\R^N}\eta\,dx=1$. Then $\mu_R*\eta_\epsilon \in C^\infty_c\cap\mathcal{P}(\R^N)$ and $\mu_R*\eta_\epsilon \rightharpoonup\mu_R$ as $\epsilon\to 0$. Here we are using implicitly the metrizability of weak convergence. Since $\mu_R*\eta_\epsilon \to \rho_R$ almost everywhere, Fatou's lemma implies that
$$
\liminf_{\epsilon\to 0} \int_{\R^N} (\mu_R*\eta_\epsilon)^q\,dx \geq \int_{\R^N} \rho_R^q\,dx\,.
$$
Moreover, since $\mu_R$ has compact support, the support of $\mu_R*\eta_\epsilon$ is contained in a bounded set independent of $\epsilon$ and therefore the interaction term is, in fact, continuous under weak convergence (see, e.g.,~\cite[Proposition 7.2]{santambrogio2015optimal}), that is,
$$
\liminf_{\epsilon\to 0} I_\lambda[\mu_R*\eta_\epsilon] = I_\lambda[\mu_R]\,.
$$
Thus, we have shown that
$$
\liminf_{\epsilon\to0} \mathcal F[\mu_R*\eta_\epsilon] \leq \mathcal G[\mu_R]\,.
$$
Hence for any $R=n\in\N$, we can find an $\epsilon_n>0$, small enough, such that $\mu_n*\eta_{\epsilon_n}\rightharpoonup\mu$ and finally obtain that
$$
\mathcal F^\Gamma[\mu]\le\lim_{n\to\infty}\mathcal F[\mu_n*\eta_{\epsilon_n}]\le\mathcal G[\mu]\,.
$$
\end{proof}

In Section~\ref{Sec:Existence}, using symmetric decreasing rearrangements, we proved that there is a minimizing sequence which converges to a minimizer. Here we have a stronger property.
\begin{proposition}
Let $N/(N+\lambda)<q<1$. Then \emph{any} minimizing sequence for $\mathcal F^\Gamma$ is relatively compact, up to translations, with respect to weak convergence. In particular, there is a minimizer for $\mathcal F^\Gamma$.
\end{proposition}
\begin{proof}
Let $(\mu_n)_{n\in\N}$ be a minimizing sequence for $\mathcal F^\Gamma$ in $\mathcal P(\R^N)$. After an $n$-dependent translation we may assume that for any $n\in\N$,
$$
\int_{\R^N} |x|^\lambda\,d\mu_n(x)=\inf_{a\in\R^N} \int_{\R^N} |x-a|^\lambda\,d\mu_n(x)
$$
according to Lemma \ref{momentbound}. Corollary~\ref{apriori} applies
$$
\int_{\R^N} |x|^\lambda\,d\mu_n(x) \leq 4\,\lambda\left( \sup_n \mathcal F^\Gamma[\mu_n]+ C\right)\,,
$$
which implies that $(\mu_n)_{n\in\N}$ is tight. By Prokhorov's theorem and after passing to a subsequence if necessary, $(\mu_n)_{n\in\N}$ converges weakly to some $\mu_*\in\mathcal P(\R^N)$. By the lower-semicontinuity property of Lemma~\ref{lsc}, we obtain that
$$
\inf_{\mu\in\mathcal{P}(\R^N)}\mathcal{F}^\Gamma[\mu]=\lim_{n\to\infty}\mathcal{F}^\Gamma[\mu_{n}]\ge\inf_{\mu\in\mathcal{P}(\R^N)}\mathcal{F}^\Gamma[\mu]\,,
$$
which concludes the proof.
\end{proof}

\begin{remark}\label{rmk:Fresc} By symmetrization, Lemma~\ref{sym-pos} and Proposition~\ref{equiv}, we learn that, up to translations, any minimizer $\mu$ of $\mathcal F^\Gamma$ is of the form $\mu=\rho+M\,\delta$, with $M\in[0,1)$ and $\rho\in\mathrm L^1_+\cap \mathrm L^q(\R^N)$. Moreover, $\rho$ is radially symmetric non-increasing and strictly positive. The minimizers of $\mathcal{F}^\Gamma$ satisfy the Euler-Lagrange conditions given by \eqref{EL}. This can be also shown by taking variations directly on $\mathcal{F}^\Gamma$ as in \cite{2018arXiv180306232C}.\end{remark}

\subsection{Uniqueness}
\begin{theorem}\label{Thm:Uniqueness} Let $N/(N+\lambda)<q<1$ and assume either that $1-1/N\le q<1$ and $\lambda\ge1$, or $2\le\lambda\le 4$. Then the minimizer of $\mathcal{F}^\Gamma$ on $\mathcal{P}(\R^N)$ is unique up to translation. \end{theorem}
Notice that Theorem~\ref{prop:lambda=2to4} is a special case of Theorem~\ref{Thm:Uniqueness}. Theorem~\ref{FreeEnergy} is a direct consequence of Proposition~\ref{equiv} and Theorem~\ref{Thm:Uniqueness}.

\begin{proof} The proof relies on the notion of displacement convexity by mass transport in the range $1-1/N\le q<1$, $\lambda\ge 1$ and on a recent convexity result,~\cite[Theorem~2.4]{lopes2017uniqueness}, of O.~Lopes in the case $2\le\lambda\le 4$. Since $N/(N+4)<1-1/N$ for $N\geq 2$, there is a range of parameters $q$ and $\lambda$ such that $N/(N+\lambda)<q<1-1/N$ and $2\le\lambda\le 4$, which is not covered by mass transport. Ranges of the parameters are shown in Fig.~\ref{Fig2}.

\smallskip\noindent$\bullet$ \emph{Displacement convexity and mass transport}. We assume that $1-1/N\le q<1$ and $\lambda\ge1$. Under these hypothesis,~\cite[Theorem~2.2]{McCann97} and~\cite[Theorem~9.4.12, p.~224]{AGS} imply that the functional $\mathcal F^\Gamma$ restricted to the set of absolutely continuous measures is strictly geodesically convex with respect to the Wasserstein-2 metric. As the minimizers might not be absolutely continuous, we cannot apply these results directly but we can adapt their proofs. We shall say that the mesurable map $T:\R^N\longrightarrow \R^N$ pushes forward the measure $\mu$ onto~$\nu$, or that \emph{$T$ transports $\mu$ onto $\nu$}, if and only if
$$
\int_{\R^N} \varphi\big(T(x)\big) \, d\mu(x) = \int_{\R^N} \varphi(x) \,d\nu(x)
$$
for all bounded and continuous functions $\varphi$ on $\R^N$. This will be written as $\nu=T\#\mu$.

Let us argue by contradiction and assume that there are two distinct radial minimizers $\mu_0=\rho_0+M_0\,\delta$ and $\mu_1=\rho_1+M_1\,\delta$, with $M_1\ge M_0$. We define
\begin{equation*}
F(s)=\mu_0(B_s)\qquad\mbox{and}\qquad G(s)=\mu_1(B_s)
\end{equation*}
on $(0,\infty)$. Both functions are monotone increasing according to Lemma~\ref{sym-pos} and Proposition~\ref{equiv}, so that they admit well defined inverses $F^{-1}:[0,1)\to[0,\infty)$ and $G^{-1}:[0,1)\to[0,\infty)$. Let $T:\R^N\to \R^N$ with
\begin{equation*}
T(x):=G^{-1}\big(F(|x|)\big)\,\frac{x}{|x|}
\end{equation*}
be the optimal transport map pushing $\mu_0$ forward onto $\mu_1$ according, \emph{e.g.}, \cite{Villani03}, which is noted as $T\#\mu_0=\mu_1$. With $s_*:=F^{-1}(M_1-M_0)$, we note that $G^{-1}\big(F(s)\big)=0$ for any $s\le s_*$ and $s\mapsto G^{-1}\big(F(s)\big)$ is strictly increasing on $(s_*,1)$. This implies that $T:B_{s_*}^c\to \R^N\setminus\{0\}$ is invertible and $\nabla T$ is positive semi-definite. We consider the midpoint of the nonlinear interpolant which is given by
\begin{equation*}
\mu_{1/2}=\tfrac12\,(I+T)\#\mu_0
\end{equation*}
where $I(x)=x$ denotes the identity map. For any $\lambda\ge 1$, we have that
\begin{align*}
 I_\lambda[\mu_{1/2}]&\displaystyle=\iint_{\R^N\times\R^N}\big|\tfrac12\,\big(x+T(x)\big)-\tfrac12\,\big(y+T(y)\big)\big|^\lambda\,d\mu_0(x)\,d\mu_0(y)\\[4mm]
&<\iint_{\R^N\times\R^N}\left(\tfrac12\,|x-y|^\lambda+\tfrac12\,\big|T(x)- T(y)\big|^\lambda\right)\,d\mu_0(x)\,d\mu_0(y)=\tfrac12\,\left(I_\lambda[\mu_0]+I_\lambda[\mu_1]\right)\,.
\end{align*}
Let $\mathrm{Id}$ be the identity matrix. By the change of variable formula as in \cite{McCann97}, we obtain that
\begin{equation*}
-\frac{1}{1-q}\int_{\R^N}\rho_{1/2}(x)^q\,dx=-\frac{1}{1-q}\int_{\R^N} \left(\frac{\rho_{0}(x)}{\det\big(\tfrac12\big(\mathrm{Id}+\nabla T(x)\big)\big)}\right)^q \det\big(\tfrac12\big(\mathrm{Id}+\nabla T(x)\big)\big)\,dx\,.
\end{equation*}
Using $q\ge1-1/N$, the fact that $\nabla T$ is positive semi-definite and the concavity of $s\mapsto\det\big((1-s)\,\mathrm{Id}+s\,\nabla T\big)^{1-q}$, we obtain that
$$
-\,\det\left(\tfrac12\,\big(\mathrm{Id}+\nabla T\big)\right)^{1-q}\le-\tfrac12\,\det(\mathrm{Id})-\tfrac12\,\det\big(\nabla T\big)^{1-q}\,.
$$
Hence
$$
-\frac{1}{1-q}\int_{\R^N}\rho_{1/2}^q\,dx\le\frac12\left(-\frac{1}{1-q}\int_{\R^N} \rho_{0}^q\,dx-\frac{1}{1-q}\int_{\R^N} \left(\frac{\rho_{0}}{\det\big(\nabla T\big)}\right)^q \det\big(\nabla T\big)\,dx\right)\,.
$$
Since $T:B_{s_*}^c\to \R^N$ is invertible and $T\#\rho_0\,\measurerestr B_{s_*}^c=\rho_1$, we can undo the change of variables:
\begin{multline*}
-\frac{1}{1-q}\int_{\R^N} \left(\frac{\rho_{0}}{\det\big(\nabla T\big)}\right)^q \det\big(\nabla T\big)\,dx= -\frac{1}{1-q}\int_{B_{s_*}^c} \left(\frac{\rho_{0}}{\det\big(\nabla T\big)}\right)^q \det\big(\nabla T\big)\,dx\\=-\frac{1}{1-q}\int_{\R^N}\rho_{1}^q\,dx\,.
\end{multline*}
Altogether, we have shown that $\mathcal{F}^\Gamma[\mu_{1/2}]<\frac12\,\left(\mathcal{F}^\Gamma[\mu_0]+\mathcal{F}^\Gamma[\mu_1]\right)$, which contradicts the assumption that $\mu_0$ and $\mu_1$ are two distinct minimizers. Notice that displacement convexity is shown only in the set of radially decreasing probability measures of the form $\mu=\rho+M\,\delta$.

\smallskip\noindent$\bullet$ \emph{Linear convexity of the functional $\F^\Gamma$}. We assume that $2\le\lambda\le 4$. Let $\mu_0=\rho_0+M_0\,\delta$ and $\mu_1=\rho_1+M_1\,\delta$ be two radial minimizers and consider the function
$$
[0,1]\ni t \mapsto\mathcal F^{\Gamma}\big[(1-t)\,\mu_0+ t\,\mu_1] =: f(t)\,.
$$
We shall prove that $f$ is strictly convex if $\mu_0\not\equiv\mu_1$. In this case, since $\mu_0$ is a minimizer, we have $f(t)\geq f(0)$ for all $0\le t\le 1$ and therefore $f'(0)\geq 0$. Together with the strict convexity this implies $f(1)>f(0)$, which contradicts the fact that $\mu_1$ is a minimizer. This is why we compute
\begin{equation*}
f''(t) = \frac{1}{\lambda}\,I_\lambda[\mu_0-\mu_1]+ q \int_{\R^N}\big((1-t)\,\rho_0+t\,\rho_1\big)^{q-2} (\rho_1-\rho_0)^2 \,dx \,.
\end{equation*}
According to~\cite[Theorem~2.4]{lopes2017uniqueness}, we have that $I_\lambda[h] \geq 0$ under the assumption $2\le\lambda\le 4$, for all $h$ such that that $\int_{\R^N}\big(1+|x|^\lambda\big)\,|h|\,dx<\infty$ with $\int_{\R^N} h\,dx = 0$ and $\int_{\R^N} x\,h\,dx = 0$. Applied with $h=\rho_0-\rho_1$, this proves the strict convexity if $M_0=M_1=0$. We have now to adapt the result of O.~Lopes to the measure valued setting, \emph{i.e.}, $(M_0,M_1)\neq(0,0)$.

Some care is needed with the second term as the power $q-2$ is negative, but since we know that the optimizers are positive a.e.~in $\R^N$ the last term in the expression of $f''(t)$ is strictly positive if $\rho_1\not\equiv\rho_0$, or eventually $+\infty$.

We have to show that $I_\lambda[\mu_0-\mu_1]>0$. If $\lambda=2$ or $\lambda=4$, the convexity follows by expanding $|x-y|^\lambda$, so we can restrict our study to $2<\lambda<4$. By Plancherel's identity we obtain that
\begin{equation*}\label{eq:aux4}
I_\lambda[\mu_0-\mu_1]=(2\pi)^\frac N2\,2^{\lambda+\frac N2}\,\frac{\Gamma\left(\frac{\lambda+N}{2}\right)}{\Gamma\left(-\frac{\lambda}{2}\right)}\,\Big\langle H_{-(N+\lambda)},\,|\hat{\mu}_0-\hat{\mu}_1|^2\Big\rangle
\end{equation*}
where $H_{-(N+\lambda)}\in \mathcal{S}'(\R^N)$ is a radial tempered distribution of homogeneity $-(N+\lambda)$. In particular, for any $\phi\in\mathcal{S}(\R^N)$ we have 
$$
\langle H_{-(N+\lambda)},\phi\rangle=\int_{\R^N}\frac1{|\xi|^{N+\lambda}}\left(\phi(\xi)-\sum_{|\alpha|\leq[\lambda]} \frac{\xi^\alpha}{\alpha!}\,\partial^\alpha\phi(0)\right)\,d\xi
$$
where $[\lambda]$ denotes the integer part of $\lambda$: see~\cite{lopes2017uniqueness,gelfand1964generalized}. These identities extend by continuity to all bounded functions $\phi\in C^2(\R^N)$ if $\lambda<3$ and $C^3(\R^N)$ if $\lambda<4$.

By Lemma~\ref{momentbound}, we know that $\int_{\R^N}|x|^\lambda\,d\mu_i(x)$ is finite for $i=0$, $1$, so that $\hat\mu_i$ is of class $C^2$ if $\lambda<3$ and of class $C^3$ if $\lambda<4$. Since $\mu_i(\R^N)=1$ and $\int_{\R^N}x\,d\mu_i=0$, we infer $\hat{\mu}_i(0)=1$ and $\nabla\hat{\mu}_i(0)=0$. This implies that $\partial^\alpha |\hat\mu_0-\hat\mu_1|^2(0)=0$ for $|\alpha|\leq 2$ if $\lambda<3$ and for $|\alpha|\leq 3$ if $\lambda<4$. We conclude that
\begin{equation*}
I_\lambda[\mu_0-\mu_1]\ge 0
\end{equation*}
with strict inequality unless $\mu_0=\mu_1$. Thus, we have shown that $f''(t)>0$ as claimed.\end{proof}

\appendix\section{Toy Model for Concentration}\label{Appendix}

Eq.~\eqref{FP} is a \emph{mean field-type} equation, in which the \emph{drift term} is an average of a spring force $\nabla W_\lambda(x)$ for any $\lambda>0$. The case $\lambda=2$ corresponds to linear springs obeying Hooke's law, while large $\lambda$ reflect a force which is small at small distances, but becomes very large for large values of $|x|$. In this sense, it is a \emph{strongly confining} force term. By expanding the diffusion term as
$\Delta\rho^q=q\,\rho^{q-1}\left(\Delta\rho+(q-1)\,\rho^{-1}\,|\nabla\rho|^2\right)$
and considering $\rho^{q-1}$ as a diffusion coefficient, it is obvious that this \emph{fast diffusion} coefficient is large for small values of $\rho$ and has to be balanced by a very large drift term to avoid a \emph{runaway} phenomenon in which no stationary solutions may exist in $\mathrm L^1(\R^N)$. In the case of a drift term with linear growth as $|x|\to+\,\infty$, it is well known that the threshold is given by the exponent $q=1-2/N$ and it is also known according to, \emph{e.g.},~\cite{MR797051} for the pure fast diffusion case (no drift) that $q=1-2/N$ is the threshold for the global existence of nonnegative solutions in $\mathrm L^1(\R^N)$, with constant mass.

In the regime $q<1-2/N$, a new phenomenon appears which is not present in linear diffusions. As emphasized in~\cite{MR2282669}, the diffusion coefficient $\rho^{q-1}$ becomes small for large values of $\rho$ and does not prevent the appearance of singularities. Let us observe that $W_\lambda$ is a convolution kernel which averages the solution and can be expected to give rise to a smooth effective potential term $V_\lambda=W_\lambda\ast\rho$ at $x=0$ if we consider a radial function~$\rho$. This is why we expect that $V_\lambda(x)=V_\lambda(0)+O\left(|x|^2\right)$ for $|x|$ small, at least for $\lambda\ge1$. With these considerations at hand, let us illustrate some consequences with a simpler model involving only a given, external potential~$V$. Assume that $u$ solves the \emph{fast diffusion with external drift} given by
$$
\partial_t u=\Delta u^q+\,\nabla\cdot\big(u\,\nabla V\big)\,.
$$
To fix ideas, we shall take $V(x)=\frac12\,|x|^2+\frac1\lambda\,|x|^\lambda$, which is expected to capture the behavior of the potential $W_\lambda\ast\rho$ at $x=0$ and as $|x|\to+\,\infty$ when $\lambda\ge2$. Such an equation admits a free energy functional
$$
u\mapsto\irN{V\,u}-\frac1{1-q}\irN{u^q}\,,
$$
whose bounded minimizers under a mass constraint on $\irN u$ are, if they exist, given by
$$
u_h(x)=\left(h+\frac{1-q}{q}V(x)\right)^{-\frac1{1-q}}\,\quad\forall\,x\in\R^N\,.
$$
A linear spring would simply correspond to a fast diffusion Fokker--Planck equation when $V(x)=|x|^2$, \emph{i.e.}, $\lambda=2$. One can for instance refer to~\cite{MR3497125} for a general account on this topic. In that case, it is straightforward to observe that the so-called \emph{Barenblatt profile} $u_h$ has finite mass if and only if $q>1-2/N$. For a general parameter $\lambda\ge2$, the corresponding integrability condition for $u_h$ is $q>1-\lambda/N$. But $q=1-2/N$ is also a threshold value for the regularity. Let us assume that $\lambda>2$ and $1-\lambda/N<q<1-2/N$, and consider the stationary solution~$u_h$, which depends on the parameter $h$. The mass of $u_h$ can be computed for any $h\ge0$ as
$$
m_\lambda(h):=\irN{\left(h+\frac{1-q}{q}V(x)\right)^{-\frac1{1-q}}}\le m_\lambda(0)=\irN{\left(\tfrac12\,|x|^2+\frac{1-q}{\lambda\,q}\,|x|^\lambda\right)^{-\frac1{1-q}}}\,.
$$
Now, if one tries to minimize the free energy under the mass contraint $\irN u=m$, it is left to the reader to check that the limit of a minimizing sequence is simply the measure $\big(m-m_\lambda(0)\big)\,\delta+u_{0}$ for any $m>m_\lambda(0)$. For the model described by Eq.~\eqref{FP}, the situation is by far more complicated because the mean field potential $V_\lambda=W_\lambda\ast\rho$ depends on the regular part $\rho$ and we have no simple estimate on a critical mass as in the case of an external potential~$V$.

\section{Other related inequalities}\label{Sec:qge1}

It is natural to ask why $q$ has been taken in the range $(0,1)$ and whether an inequality similar to~\eqref{ineq:rHLS} holds for $q\ge1$. The free energy approach of Section~\ref{sec:freeenergy} provides simple guidelines to distinguish a fast diffusion regime with $q<1$ from a \emph{porous medium} regime with $q>1$ and a linear diffusion regime with $q=1$ exactly as in the case of the Gagliardo-Nirenberg inequalities associated with the classical fast diffusion or porous medium equations and studied in~\cite{MR1940370}. 
\begin{theorem}\label{mainPM} Let $N\geq 1$, $\lambda>0$ and $q\in(1,+\infty)$. Then the inequality
\be{ineq:rHLS-PM}
I_\lambda[\rho]\left(\int_{\R^N}\rho(x)^q\,dx\right)^{(\alpha-2)/q}\geq\hspace*{4pt}\mathcal C_{N,\lambda,q}\left(\int_{\R^N}\rho(x)\,dx\right)^\alpha
\ee
holds for any nonnegative function $\rho\in\mathrm L^1\cap\mathrm L^q(\R^N)$, for some positive constant $\mathcal C_{N,\lambda,q}$. Moreover, a radial positive, non-increasing, bounded function $\rho\in\mathrm L^1\cap\mathrm L^q(\R^N)$ with compact support achieves the equality case.\end{theorem}
Compared to~\eqref{ineq:rHLS} with $2\,N/(2\,N+\lambda)<q<1$, notice that, as in the case of Gagliardo-Nirenberg inequalities, the position of $\int_{\R^N}\rho\,dx$ and $\int_{\R^N}\rho^q\,dx$ have been interchanged in the inequality. As in the case $q<1$, the exponent $\alpha$ is given by
\[
\alpha=\frac{q\,(2\,N+\lambda)-2\,N}{N\,(q-1)}
\]
and takes values in $(2+\lambda/N,+\infty)$ in the range $q>1$.
\begin{proof} For any nonnegative function $\rho\in\mathrm L^1\cap\mathrm L^q(\R^N)$ we have
\[
I_\lambda[\rho]\ge\mathcal C_*\left(\int_{\R^N}\rho(x)^\frac{2\,N}{2\,N+\lambda}\,dx\right)^{2+\frac\lambda N}
\]
with $\mathcal C_*=\mathcal C_{N,\lambda,2\,N/(2\,N+\lambda)}$ by Theorem~\ref{main}. By H\"older's inequality,
\[
\left(\int_{\R^N}\rho(x)^\frac{2\,N}{2\,N+\lambda}\,dx\right)^{2+\frac\lambda N}\left(\int_{\R^N}\rho(x)^q\,dx\right)^{(\alpha-2)/q}\geq\left(\int_{\R^N}\rho(x)\,dx\right)^\alpha\,,
\]
and so
\be{Interp}
I_\lambda[\rho]\left(\int_{\R^N}\rho(x)^q\,dx\right)^{(\alpha-2)/q}\geq\mathcal C_*\left(\int_{\R^N}\rho(x)\,dx\right)^\alpha\,.
\ee
This proves~\eqref{ineq:rHLS-PM} for some constant $\mathcal C_{N,\lambda,q}\ge\mathcal C_*$. The existence of a radial non-increasing minimizer is an easy consequence of rearrangement inequalities, Helly's selection theorem and Lesgue's theorem of dominated convergence as in the proof of Proposition~\ref{opt0}. We read from the Euler--Lagrange equation
\[
2\,\frac{{\int_{\R^N}|x-y|^\lambda\,\rho(y)\,dy}}{I_\lambda[\rho]}+\frac{(\alpha-2)\,\rho(x)^{q-1}}{\int_{\R^N} \rho(y)^q\,dy}-\frac\alpha{\int_{\R^N} \rho(y)\,dy}=0\,,
\]
that $\rho$ has compact support. Indeed, because of the constraint $\rho\ge0$, the equation is restricted to the interior of the support of $\rho$, which is either a ball or $\R^N$. Then we can use the Euler-Lagrange equation to write
\[
\rho(x)=\left(C_1-C_2\int_{\R^N}|x-y|^\lambda\,\rho(y)\,dy\right)_+^{1/(q-1)}
\]
for some positive constants $C_1$ and $C_2$, and since $\int_{\R^N}|x-y|^\lambda\,\rho(y)\,dy\sim|x|^\lambda\int_{\R^N}\rho(y)\,dy$ as $|x|\to+\infty$, the support of $\rho$ has to be a finite ball by integrability of $\rho$.\end{proof}

As in Proposition~\ref{equiv}, we notice that the free energy functional also defined in the case $q>1$ by $\mathcal F[\rho]:=\frac1{q-1}\int_{\R^N}\rho^q\,dx+\frac{1}{2\lambda}\,I_\lambda[\rho]$ is bounded from below by an optimal constant which can be computed in terms of the optimal constant $\mathcal C_{N,\lambda,q}$ in~\eqref{ineq:rHLS-PM} by a simple scaling argument.

At the threshold of the porous medium and fast diffusion regimes, there is a \emph{linear regime} corresponding to $q=1$. If we consider the limit of $\mathcal F[\rho]-\frac1{q-1}\int_{\R^N}\rho\,dx$ as $q\to1$, we see that the limiting free energy takes the standard form $\rho\mapsto\int_{\R^N}\rho\,\log\rho\,dx+\frac{1}{2\lambda}\,I_\lambda[\rho]$, which is bounded from below according to the following \emph{logarithmic Sobolev} type inequality.
\begin{theorem}\label{mainLog} Let $N\geq 1$ and $\lambda>0$. Then the inequality
\be{ineq:rHLS-log}
\int_{\R^N}\rho\,\log\rho\,dx +\frac N\lambda\,\log\left(\frac{I_\lambda[\rho]}{\mathcal C_{N,\lambda,1}}\right) \ge 0
\ee
holds for any nonnegative function $\rho\in\mathrm L^1(\R^N)$ such that $\int_{\R^N}\rho(x)\,dx=1$ and $\rho\log\rho\in\mathrm L^1(\R^N)$, for some positive constant $\mathcal C_{N,\lambda,1}$. Moreover, a radial positive, non-increasing, bounded function $\rho\in\mathrm L^1\cap\mathrm L^q(\R^N)$ achieves the equality case.\end{theorem}
\begin{proof} With $\varepsilon=1/\alpha$, taking the $\log$ on both sides of~\eqref{Interp} and multiplying by $\varepsilon$ yields
\[
g(\varepsilon):=\varepsilon\,\log\left(\frac{I_\lambda[\rho]}{\mathcal C_*}\right)+\frac{1-2\,\varepsilon}q\,\log\left(\int_{\R^N}\rho^q\,dx\right)-\log\left(\int_{\R^N}\rho\,dx\right)\ge0\,.
\]
Since $q(\varepsilon)=1+\frac\lambda N\,\varepsilon+O(\varepsilon^2)$ for small $\varepsilon>0$, we obtain $g(0)=0$ in the limit, and the first order term is nonnegative for small enough $\varepsilon$,
\[
g'(0)=\log\left(\frac{I_\lambda[\rho]}{\mathcal C_*}\right)-\frac{2\,N+\lambda}N\,\log\left(\int_{\R^N}\rho\,dx\right)+\frac\lambda N\,\frac{\int_{\R^N}\rho\,\log\rho\,dx}{\int_{\R^N}\rho\,dx}\ge0\,.
\]
Hence there exists an optimal constant $\mathcal C_{N,\lambda,1}\ge\mathcal C_*$ such that
\[
\int_{\R^N}\rho\,\log\rho\,dx +\frac N\lambda\,\left(\int_{\R^N}\rho\,dx\right)\,\log\left(\frac{I_\lambda[\rho]}{\mathcal C_{N,\lambda,1}\left(\int_{\R^N}\rho(x)\,dx\right)^\frac{2\,N+\lambda}N}\right)\ge 0\,.
\]
and~\eqref{ineq:rHLS-log} follows by taking into account the normalization.
\end{proof}

\newpage\subsection*{Acknowledgments}\begin{spacing}{0.75}{\noindent\scriptsize
This research has been partially supported by the projects \emph{EFI}, contract~ANR-17-CE40-0030 (J.D.) and \emph{Kibord}, contract~ANR-13-BS01-0004 (J.D., F.H.) of the French National Research Agency (ANR), and by the U.S. National Science Foundation through grant DMS-1363432 (R.L.F.). J.D.~thanks M.~Zhu for references to the literature of Carlson type inequalities. The research stay of F.H.~in Paris in December 2017 was partially supported by the Simons Foundation and by Mathematisches Forschungsinstitut Oberwolfach. J.A.C. and M.G.D. were partially supported by EPSRC grant number EP/P031587/1. R.L.F.~thanks the University Paris-Dauphine for hospitality in February 2018. The authors are very grateful to the Mittag-Leffler Institute for providing a fruitful working environment during the special semester \emph{Interactions between Partial Differential Equations \& Functional Inequalities}. They also wish to thank an anonymous referee for very detailed and useful comments which contributed to improve the paper.
\\[2pt]
\copyright\,2019 by the authors. This paper may be reproduced, in its entirety, for non-commercial purposes.}\end{spacing}\vspace*{-0.25cm}


\clearpage\vspace*{-1cm}
\setlength\unitlength{1cm}
\begin{figure}[!ht]\begin{center}\begin{picture}(12,9)
\put(0,0){\includegraphics[width=12cm]{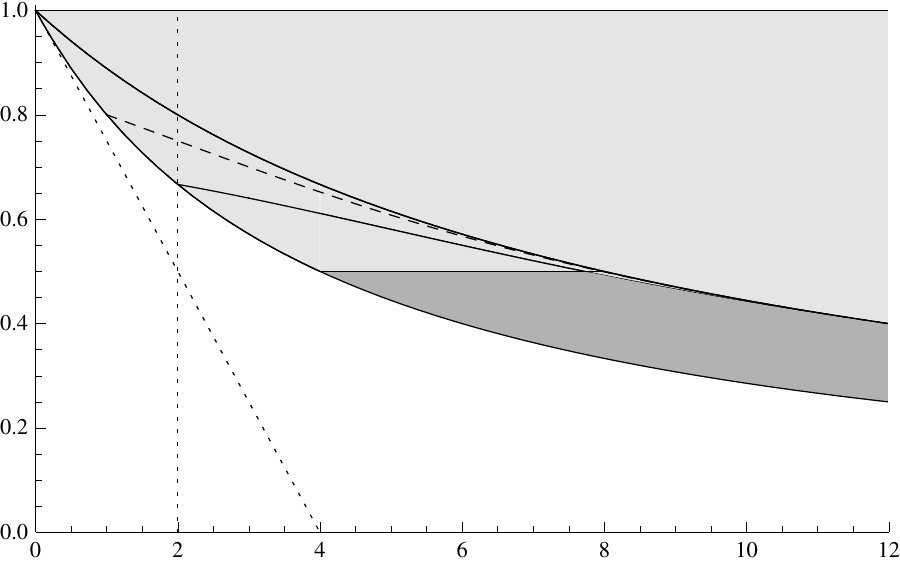}}
\put(6,5){\includegraphics[width=7cm]{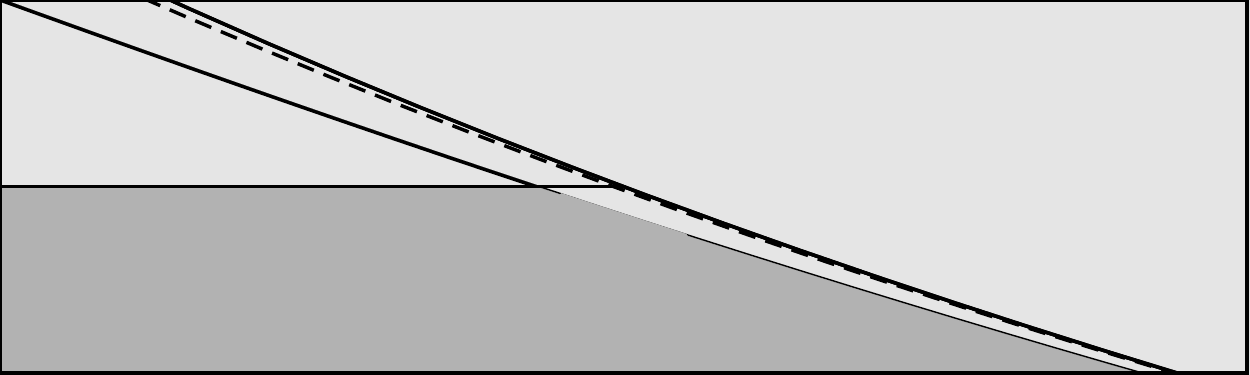}}
\put(11.5,0.5){$\lambda$}
\put(-0.25,6.5){$q$}
\put(0.8,3.7){$q=\frac{N-2}N$}
\put(3.8,4.25){$\scriptstyle q=\bar q(\lambda,N)$}
\put(10.15,3.7){$q=\frac{2\,N}{2\,N+\lambda}$}
\put(5.1,2.7){$q=\frac N{N+\lambda}$}
\end{picture}
\caption{\label{Fig1}\scriptsize Main regions of the parameters (here $N=4$), with an enlargement of the region inside the black rectangle. The case $q=2\,N/(2\,N+\lambda)$ corresponding to $\alpha=0$ has already been treated in~\cite{DZ15,Beckner2015,MR3666824}. Inequality~\eqref{ineq:rHLS} holds with a positive constant $\mathcal C_{N,\lambda,q}$ if $q>N/(N+\lambda)$, \emph{i.e.}, $\alpha<1$, which determines the admissible range corresponding to the grey area, and it is achieved by a function $\rho$ (without any Dirac mass) in the light grey area. The dotted line is $q=1-\lambda/N$: it is tangent to the admissible range of parameters at $(\lambda,q)=(0,1)$, and it is also the threshold line for integrable stationary solutions in the toy model in the Appendix. In the dark grey region, Dirac masses with $M_*>0$ are not excluded. The dashed curve corresponds to the curve $q=2\,N\,\big(1-2^{-\lambda}\big)\big/\big(2\,N\big(1-2^{-\lambda}\big)+\lambda\big)$ and can hardly be distinguished from $q=2\,N/(2\,N+\lambda)$ when~$q$ is below $1-2/N$. The curve $q=\bar q(\lambda,N)$ of Corollary~\ref{Cor:Mstar} is also represented. Above this curve, no Dirac mass appears when minimizing the relaxed problem corresponding to~\eqref{ineq:rHLS}. Whether Dirac masses appear in the region which is not covered by Corollary~\ref{Cor:Mstar} is an open question.}
\end{center}\end{figure}
\begin{figure}[!hb]\begin{center}\begin{picture}(15,4)
\put(0,0){\includegraphics[width=7.5cm]{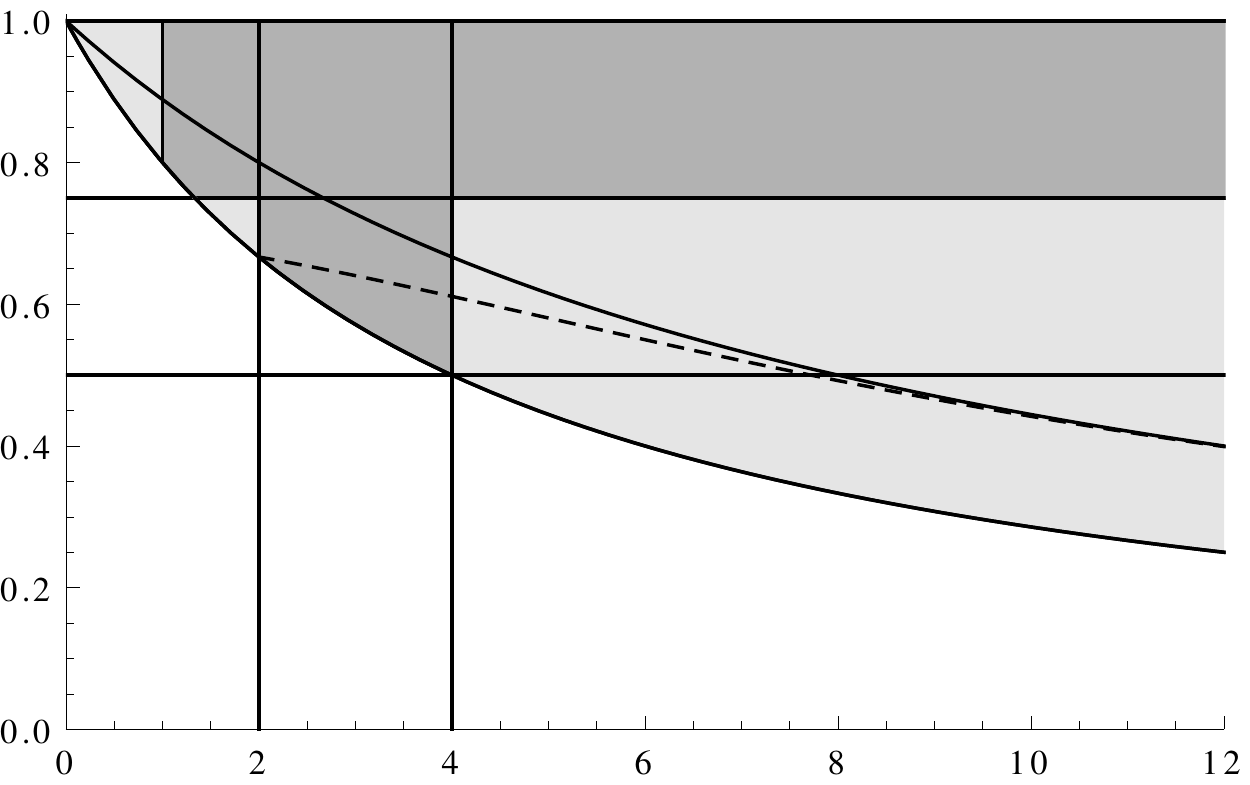}\hspace*{0,25cm}\includegraphics[width=7.7cm]{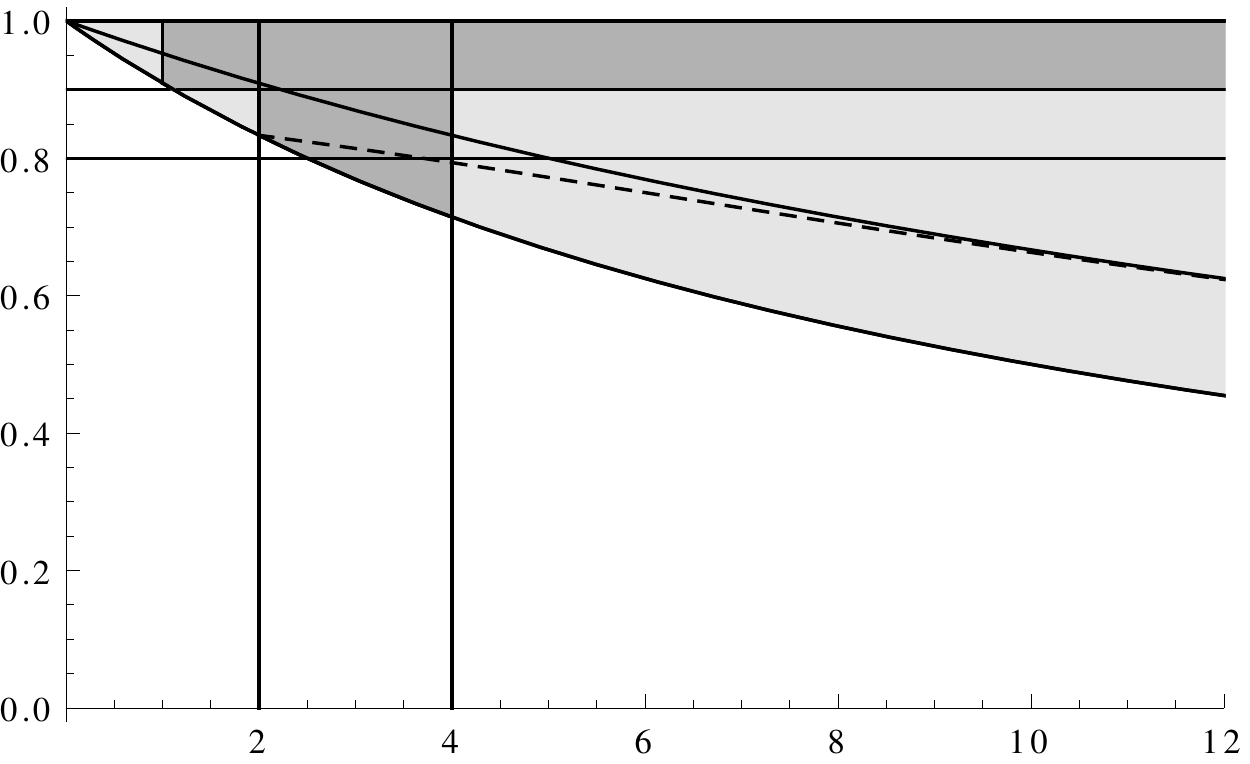}}\end{picture}
\caption{\label{Fig2}\scriptsize Darker grey areas correspond to regions of the parameters $(\lambda,q)\in(0,+\infty)\times[0,1)$ for which there is uniqueness of the measure-valued minimizer, with $N=4$ (left) and $N=10$ (right). The dashed curve is $q=\bar q(\lambda,N)$, above which minimizers are bounded, with no Dirac singularity. Horizontal lines correspond to $q=0$, $1-2/N$, $1-1/N$ and~$1$.}
\end{center}\end{figure}

\end{document}